\documentclass{amsart}
\usepackage{enumerate}
\usepackage{amsmath}
\usepackage{amssymb,latexsym}
\usepackage{amsthm}
\usepackage{color}
\usepackage{cancel}
\usepackage{graphicx}

\usepackage{hyperref}
\usepackage{comment}
\usepackage{amscd}

\DeclareMathOperator{\C}{\mathcal{C}}

\newtheorem{theorem}{Theorem}[section]

\newtheorem{lemma}[theorem]{Lemma}

\newtheorem{definition}[theorem]{Definition}
\newtheorem{proposition}[theorem]{Proposition}

\newtheorem{remark}[theorem]{Remark}

\newcommand{\fqn}{\mathbb{F}_{q^n}}

\newcommand{\cC}{{\mathcal C}}

\newcommand{\KK}{{\mathbb K}}

\newcommand{\fq}{{\mathbb F}_{q}}
\newcommand{\fqq}{{\mathbb F}_{q^4}}

\title{On a family of linear MRD codes with parameters $[8\times8,16,7]_q$}

\begin{small}

\author[M. Timpanella]{Marco Timpanella}
\address{Marco Timpanella, \textnormal{Universit\`a degli Studi di Perugia, Dipartimento di Matematica e Informatica, Via Vanvitelli, 1, 06123 Perugia PG, Italy}}
\email{marco.timpanella@unipg.it}
\author[G. Zini]{Giovanni Zini}
\address{Giovanni Zini, \textnormal{Universit\`a degli Studi di Modena e Reggio Emilia, Dipartimento di Scienze Fisiche, Informatiche e Matematiche, Via Campi, 213/b Modena MO, Italy}}
\email{giovanni.zini@unimore.it}
\end{small}

\subjclass{14G50,11T06}
\keywords{MRD codes, linearized polynomials, algebraic varieties}

\begin{document}

\maketitle

\begin{abstract}
In this paper we consider a family $\mathcal{F}$ of $2n$-dimensional $\mathbb{F}_q$-linear rank metric codes in $\mathbb{F}_q^{n\times n}$ arising from polynomials of the form $x^{q^s}+\delta x^{q^{\frac{n}{2}+s}}\in\mathbb{F}_{q^n}[x]$. The family $\mathcal{F}$ was introduced by Csajb\'ok, Marino, Polverino and Zanella (2018) as a potential source for maximum rank distance (MRD) codes. Indeed, they showed that $\mathcal{F}$ contains MRD codes for $n=8$, and other subsequent partial results have been provided in the literature towards the classification of MRD codes in $\mathcal{F}$ for any  $n$. In particular, the classification has been reached when $n$ is smaller than $8$, and also for $n$ greater than $8$ provided that $s$ is small enough with respect to $n$.
In this paper we deal with the open case $n=8$, providing a classification for any large enough odd prime power $q$.
The techniques are from algebraic geometry over finite fields, since our strategy requires the analysis of certain $3$-dimensional $\mathbb{F}_q$-rational algebraic varieties in a $7$-dimensional projective space.
We also show that the MRD codes in $\mathcal{F}$ are not equivalent to any other MRD codes known so far.
\end{abstract}

\section{Introduction}
Rank metric codes over finite fields were introduced by Delsarte in its seminal paper \cite{Delsarte} in 1978, and increasingly studied since then.
The interest is mainly boosted by their applications in information theory, such as crisscross error correction \cite{Roth}, code-based cryptography \cite{GPT} and linear network coding \cite{SKK}, but also by the connections with other mathematical objects, such as semifields \cite{John} and linear sets \cite{ZZ}.

An \emph{$\fq$-linear rank metric code} $\mathcal{C}$ is an $\fq$-vector subspace of the space $\mathbb{F}_{q}^{m\times n}$ of $m\times n$ matrices over the finite field $\mathbb{F}_q$, equipped with the rank distance $d(A,B):={\rm rank}(A-B)$.
We denote the main parameters of $\mathcal{C}$ by $[m\times n,k,d]_q$, where $k$ is the $\fq$-dimension of $\mathcal{C}$ and $d$ is the minimum distance of $\mathcal{C}$, i.e. the minimum rank distance between two distinct elements of $\cC$.
The Singleton-like bound $k\leq \max\{m,n\}(\min\{m,n\}-d+1)$ holds; see \cite{Delsarte}. When equality holds, $\cC$ is called \emph{maximum rank distance} (MRD for short). Codes of this sort have particular interest because of the optimality of their parameters.
The first examples of MRD codes were constructed by Delsarte \cite{Delsarte} and independently by Gabidulin \cite{Gabidulin}, and are known as \emph{Gabidulin codes}.


We are interested in the square case $m=n$.
In this case, we can identify $\mathbb{F}_q^{n\times n}$ with the $\fq$-algebra $\mathcal{L}_{n,q}$ of $q$-polynomials over $\fqn$ of degree smaller than $q^n$, with composition modulo $x^{q^n}-x$; see \cite{PZ}. Thus, $\fq$-linear $[n\times n,k,d]_q$ rank metric codes $\mathcal{C}$ can be seen as $k$-dimensional $\fq$-vector subspaces of $\mathcal{L}_{n,q}$ with minimum rank distance $d$.
If in addition $\mathcal{C}$ is an $\fqn$-vector subspace of $\mathcal{L}_{n,q}$, we say that $\mathcal{C}$ is $\fqn$-linear.

Most of the known families of $\fq$-linear MRD codes are indeed $[n\times n,2n,d]_q$ $\fqn$-linear codes $\mathcal{C}_f$ for some $q$-polynomial $f(x)\in\mathcal{L}_{n,q}$, where
\[
\mathcal{C}_{f}:=\langle x,f(x)\rangle_{\fqn}=\left\{ax+bf(x)\colon a,b\in\fqn\right\}\subseteq\mathcal{L}_{n,q};
\]
see e.g. \cite{NSZ} and the references therein.
Such codes are strictly connected with so-called linear sets of the projective line ${\rm PG}(1,q^n)$ over $\fqn$. When $\mathcal{C}_f$ is MRD, the polynomial $f(x)$ is said to be \emph{scattered} and defines a \emph{scattered linear set} in ${\rm PG}(1,q^n)$; see \cite{PZ}.

Csajb\'ok et al. \cite{CMPZ} introduced for $n$ even the family of rank metric codes $\mathcal{C}_{\delta,s}:=\mathcal{C}_{f_{\delta,s}}$ with
\[
f_{\delta,s}(x):= x^{q^s}+\delta x^{q^{\frac{n}{2}+s}}\in\mathcal{L}_{n,q},
\]
where $\delta\in\mathbb{F}_{q^n}$, $\delta\ne0$, and $s\in\{1,\ldots,n-1\}$ is coprime with $n/2$.
Whenever $\delta^{1+q^{n/2}}\ne1$, the minimum distance of $\mathcal{C}_{\delta,s}$ is large: either $d=n-1$, i.e. $\mathcal{C}_{\delta,s}$ is MRD; or $d=n-2$, i.e. $\mathcal{C}_{\delta,s}$ is as close as possible to being MRD. In the latter case $\mathcal{C}_{\delta,s}$ is called Almost MRD, in analogy with Almost MDS codes in the Hamming metric; see \cite[Definition 3.1]{DLC}.

Several papers have provided partial results towards the classification of MRD codes of type $\mathcal{C}_{\delta,s}\subseteq\mathcal{L}_{n,q}$.
A sufficient condition was given in the same paper \cite{CMPZ} when $n=8$ and $q$ is odd: if $\delta^{1+q^4}=-1$, then $\mathcal{C}_{\delta,s}\subseteq\mathcal{L}_{8,q}$ is MRD.
For smaller values of $n$, a characterization of MRD codes $\mathcal{C}_{\delta,s}$ is known; the case $n=6$ has been dealt with in \cite{BCM} and \cite{PZnumber}.
For higher values of $n$, a characterization of MRD codes $\mathcal{C}_{\delta,s}$ has been obtained in \cite{PZZ} when $n$ is large enough, namely when
\[
n\geq\begin{cases}
        8s+4 & \mbox{if }q=3\mbox{ and }s>1,\mbox{ or }q=2\mbox{ and }s>2;\\
        8s+2 & \mbox{otherwise}.
\end{cases}
\]
For instance, this rules out any $n\geq10$ when $s=1$.

In this paper, we focus on the open case $n=8$.
Our main result is the characterization of the MRD codes of type $\mathcal{C}_{\delta,s}\subseteq\mathcal{L}_{8,q}$ under the assumption that $q$ is odd and large enough.
We also show that such MRD codes are not equivalent to any other known ones.
Our results are summarized as follows (the notions of equivalence and idealisers will be given in Section \ref{sec:prelim}).

\begin{theorem}\label{th:main}
Let $q\geq 1039891$ be odd, $\delta\in\mathbb{F}_{q^8}^*$, $s\in\{1,3,5,7\}$, and $f_{\delta,s}=x^{q^s}+\delta x^{q^{4+s}}\in\mathcal{L}_{8,q}$.
Then
\[
\mathcal{C}_{\delta,s}=\langle x,\,f_{\delta,s}(x)\rangle_{\mathbb{F}_{q^8}} \subseteq \mathcal{L}_{8,q}
\]
is MRD if and only if $\delta^{1+q^4}=-1$.

The $[8\times8,16,7]_q$ MRD codes $\mathcal{C}_{\delta,s}$ have left idealiser isomorphic to $\mathbb{F}_{q^8}$, right idealiser isomorphic to $\mathbb{F}_{q^4}$, and are not equivalent to the other MRD codes known so far in the literature.
\end{theorem}

Notice that Theorem \ref{th:main} is a partial answer to Conjecture 4.6 in \cite{PZZ}.
The techniques that we use to prove the first part of Theorem \ref{th:main} are from algebraic geometry over finite fields.
In particular, our starting point is an $\mathbb{F}_{q^{n/2}}$-rational plane curve $\mathcal{X}_{\delta,s}$ introduced in \cite{PZZ}; see Equation \eqref{eq:curveqodd_withs} below. For $n\geq10$, the degree of $\mathcal{X}_{\delta,s}$ is low with respect to the size of $\mathbb{F}_{q^{n/2}}$, and hence the Hasse-Weil lower bound on the number of $\mathbb{F}_{q^{n/2}}$-rational points of $\mathcal{X}_{\delta,s}$ was enough in \cite{PZZ} to deduce results on the MRD property for $\mathcal{C}_{\delta,s}$.
For $n=8$ this is not sufficient. We then translate the MRD property for $\mathcal{C}_{\delta,s}$ into the estimate of the number of $\mathbb{F}_{q}$-rational points of another $\mathbb{F}_{q}$-rational algebraic variety $\mathcal{W}$ of low degree. We investigate the absolutely irreducible components of $\mathcal{W}$ and apply the Lang-Weil lower bound.
In the study of MRD codes $\mathcal{C}_f$, similar algebraic geometric arguments have been already used in the literature, but only to deal with algebraic curves or hypersurfaces; see for instance \cite{MontanucciZanella}.
On the contrary, $\mathcal{W}$ turns out to be a $3$-dimensional variety in a $7$-dimensional projective space.

The paper is organized as follows.
Section \ref{sec:prelim} contains some preliminary results: after some generalities on algebraic varieties (Section \ref{sec:alggeom}) and codes (Section \ref{sec:codes}), we present what is known about codes $\mathcal{C}_{\delta,s}$ (Section \ref{sec:ourcodes}) and the curves $\mathcal{X}_{\delta,s}$ (Section \ref{sec:curve}).
Section \ref{sec:MRD} proves the first part of Theorem \ref{th:main} about the characterization of MRD codes $\mathcal{C}_{\delta,s}$; the proof of some technical lemmas is postponed to Section \ref{sec:lemmaqodd}.
Section \ref{sec:parequiv} completes the proof of Theorem \ref{th:main} with the investigation of parameters and equivalences.
Finally, we list in \ref{sec:open} some open problems on codes $\mathcal{C}_{\delta,s}$.

\section{Preliminaries}\label{sec:prelim}

For the rest of the paper, $q$ is an odd prime power.

\subsection{Algebraic varieties}\label{sec:alggeom}

We recall some basic facts on algebraic varieties, and refer the reader to \cite{Sha} and \cite{Sti} for a detailed introduction to varieties and function fields. A good reference for the applications of algebraic geometric techniques to polynomials over finite field is \cite{Bartoli}.

We denote by $\mathbb{K}=\overline{\mathbb{F}}_q$ the algebraic closure of $\fq$, and by $\mathbb{P}^N=\mathbb{P}^N(\mathbb{K})$ the $N$-dimensional projective space over $\mathbb{K}$.
We will be interested in determining \emph{dimension} and \emph{degree} of  \emph{varieties} $\mathcal{V}$, i.e. of projective, possibly reducible, algebraic varieties $\mathcal{V}\subseteq\mathbb{P}^N$.
A variety $\mathcal{V}$ which is irreducible over $\KK$, is said to be \emph{absolutely irreducible}, and can be studied through its \emph{function field} $\KK(\mathcal{V})$ over $\KK$.
If the ideal of $\mathcal{V}$ is generated by polynomials over $\fq$, we say that $\mathcal{V}$ is \emph{defined over $\fq$}, or simply \emph{$\fq$-rational}.
When $\mathcal{V}$ is both $\fq$-rational and absolutely irreducible, $\fq(\mathcal{V})$ is its function field over $\fq$.

The following proposition recalls some elementary facts on the dimension of a variety.

\begin{proposition}\label{prop:sha}
Let $\mathcal{V},\mathcal{V}_1,\mathcal{V}_2$ be algebraic varieties in $\mathbb{P}^N$.
\begin{itemize}
    \item[(i)] \cite[Sec. 6.2, Cor. 5]{Sha} If $\mathcal{V}$ is defined by $r$ equations, then the dimension of $\mathcal{V}$ is at least $N-r$.
    \item[(ii)] \cite[Sec. 6.2, Cor. 4]{Sha} If $s$ is the maximum dimension of a linear space of $\mathbb{P}^N$ disjoint from $\mathcal{V}$, then the dimension of $\mathcal{V}$ is $N-s-1$.
    \item[(iii)] \cite[Sec. 6.2, Th. 6]{Sha}
    If $\mathcal{V}_1$ and $\mathcal{V}_2$ are absolutely irreducible and have dimension $m_1$ and $m_2$ respectively, then any non-empty component of $\mathcal{V}_1\cap\mathcal{V}_2$ has dimension at least $m_1+m_2-N$.
\end{itemize}
\end{proposition}

An estimate for the number of $\fq$-rational points of an $\fq$-rational absolutely irreducible variety is provided by the Lang-Weil bound, which is a generalization to higher dimension of the Hasse-Weil bound for curves.
We will use the Lang-Weil bound in the following improved version, due to Cafure and Matera.

\begin{proposition}\cite[Theorem 7.1]{CafureMatera}\label{Th:CafureMatera}
Let $\mathcal{V}\subseteq\mathbb{P}^N$ be an absolutely irreducible variety defined over $\mathbb{F}_q$, of dimension $m$ and degree $d$.
Let $A_q$ be the number of $\fq$-rational affine points of $\mathcal{V}$.
If $q>2(m+1)d^2$, then 
$$|A_q-q^m|\leq (d-1)(d-2)q^{m-\frac{1}{2}}+5d^{\frac{13}{3}} q^{m-1}.$$
\end{proposition}


\subsection{$q$-polynomials and rank metric codes}\label{sec:codes}

We now recall some preliminary results on $q$-polynomials and rank metric codes; for a detailed introduction to this topic we refer the reader to \cite{JohnSurvey} and \cite{PZ}.

Let $\mathcal{L}_{n,q}=\left\{\sum_{i=0}^{n-1}a_i x^{q^i}\colon a_i\in\mathbb{F}_{q^n}\right\}$ be the set of $q$-polynomials over $\fqn$ of $q$-degree smaller than $n$. Then $\mathcal{L}_{n,q}$ is an $n$-dimensional $\fq$-vector space with the usual sum and scalar multiplication, and also an $\fq$-algebra with the composition modulo $x^{q^n}-x$.
We identify polynomials in $\mathcal{L}_{n,q}$ with the associated $\fq$-linear polynomial maps over $\fq$, so that we can consider the rank and the kernel of polynomials in $\mathcal{L}_{n,q}$. This identification is an isomorphism of $\fq$-algebras between $\mathcal{L}_{n,q}$ and the space of $\fq$-endomorphisms of $\fqn$. Therefore, via the choice of an $\fq$-basis of $\fqn$, $\mathcal{L}_{n,q}$ is isomorphic as an $\fq$-algebra to the space $\mathbb{F}_q^{n\times n}$ of $n\times n$ matrices over $\fq$.
In this correspondence, the rank of a matrix equals the rank of the associated $q$-polynomial.
Therefore, $\mathcal{L}_{n,q}$ is a metric space with the rank metric $d(f,g):={\rm rank}(f-g)$, and $\fq$-linear $[n\times n,k,d]_q$ rank metric codes $\mathcal{C}$ are $k$-dimensional $\fq$-linear subspaces of $\mathcal{L}_{n,q}$ with minimum rank distance $d$ between two distinct elements of $\mathcal{C}$.
Since $\mathcal{C}$ is $\fq$-linear, $d$ equals the minimum rank of a non-zero element of $\mathcal{C}$. The Singleton-like bound $k\leq n(n-d+1)$ holds; see \cite{Delsarte}. The code $\mathcal{C}$ is \emph{maximum rank distance} (MRD for short) when it attains equality in the Singleton-like bound.

The following equivalence between rank metric codes preserves the parameters of a code; see \cite{John}.

\begin{definition}
Two $\fq$-linear $[n\times n,k,d]_q$-codes $\mathcal{C},\mathcal{C}^\prime$ are \emph{equivalent} if there exist two invertible polynomials $f_1(x),f_2(x)\in\mathcal{L}_{n,q}$ and a field automorphism $\varphi\in{\rm Aut}(\mathbb{F}_{q^n})$ such that
\[
\mathcal{C}^{\prime}=f_1\circ\mathcal{C}^{\varphi}\circ f_2=\left\{f_1\circ f^{\varphi}\circ f_2\colon f\in\mathcal{C} \right\},
\]
where $f^{\varphi}(x)$ is obtained from $f(x)$ by applying $\varphi$ to its coefficients.
\end{definition}

Related to the equivalence issue for rank metric codes is the tool of idealisers, as defined in \cite{LN}.
\begin{definition}
The \emph{left idealiser} $L(\mathcal{C})$ and \emph{right idealiser} $R(\mathcal{C})$ of an $\fq$-linear rank metric code $\mathcal{C}\subseteq\mathcal{L}_{n,q}$ are the following $\fq$-algebras:
\[
L(\mathcal{C})=\left\{h(x)\in\mathcal{L}_{n,q}\,\colon\, h\circ f\in\mathcal{C}\mbox{ for all }f\in\mathcal{C}\right\},\quad
R(\mathcal{C})=\left\{h(x)\in\mathcal{L}_{n,q}\,\colon\, f\circ h\in\mathcal{C}\mbox{ for all }f\in\mathcal{C}\right\}.
\]
\end{definition}

Idealisers are invariant under equivalence.

\begin{proposition}\cite[Proposition 4.1]{LTZ}
If $\mathcal{C},\mathcal{C}^{\prime}\subseteq\mathcal{L}_{n,q}$ are equivalent codes, then $L(\mathcal{C})\cong L(\mathcal{C}^{\prime})$ and $R(\mathcal{C})\cong R(\mathcal{C}^{\prime})$ as $\fq$-algebras.
\end{proposition}

More structural information is known for idealisers of MRD codes.

\begin{proposition}\cite[Section 5]{LTZ} \label{prop:idealisers}
If $\mathcal{C}\subseteq\mathcal{L}_{n,q}$ is an MRD code, then $L(\C)$ and $R(\C)$ are both finite fields. Also, $|L(\C)|=q^{\ell}$ and $|R(\C)|=q^m$ for some divisors $\ell$ and $m$ of $n$.
\end{proposition}

\subsection{Codes $\mathcal{C}_f$ and $\mathcal{C}_{\delta,s}$}\label{sec:ourcodes}

For any $f(x)\in\mathcal{L}_{n,q}$ of degree greater than $1$, define the $\fq$-linear code $\mathcal{C}_f$ and the $\fq$-linear space $U_f$ by
\[
\mathcal{C}_f:=\langle x,f(x)\rangle_{\fqn}=\{ax+bf(x)\colon a,b\in\fqn\}\subseteq\mathcal{L}_{n,q},
\qquad U_f:=\left\{(x,f(x))\colon x\in\mathbb{F}_{q^n} \right\}\subseteq\mathbb{F}_{q^n}\times\mathbb{F}_{q^n}.
\]
Since $\mathcal{C}_f$ has dimension $2$ over $\mathbb{F}_{q^n}$, the MRD property for $\mathcal{C}_f$ reads as follows.
\begin{remark}\label{rem:corrisp}
$\mathcal{C}_f$ is MRD if and only if $\dim_{\fq}\ker(g(x))\leq 1$ for any non-zero $g(x)\in\mathcal{C}_f$.
\end{remark}

Clearly, $\left\{\alpha x\colon \alpha\in\fqn\right\}\subseteq L(\mathcal{C}_f)$, and by Proposition \ref{prop:idealisers} equality holds if $\mathcal{C}_f$ is MRD.
Indeed, Proposition \ref{prop:caratt} shows that the parameters and the maximum size of the left idealiser are enough to characterize up to equivalence the family of codes $\mathcal{C}_f$ with scattered polynomials $f(x)$.
\begin{proposition}\label{prop:caratt}\cite[Proposition 6.1]{CMPZ}
If $\mathcal{C}\subseteq\mathcal{L}_{n,q}$ is an $[n\times n,2n,n-1]_q$ MRD code with $L(\mathcal{C})\cong\fqn$, then $\mathcal{C}$ is equivalent to $\mathcal{C}_f$ for some scattered polynomial $f(x)\in\mathcal{L}_{n,q}$.
\end{proposition}

Let $n$ be an even positive integer, $s\in\{1,\ldots,n-1\}$ be coprime with $n/2$, and $\delta\in\mathbb{F}_{q^n}^*$. Define
\[
f_{\delta,s}(x):= x^{q^s}+\delta x^{q^{\frac{n}{2}+s}}\in\mathcal{L}_{n,q},\qquad U_{\delta,s}:=U_{f_{\delta,s}},\qquad \mathcal{C}_{\delta,s}:=\mathcal{C}_{f_{\delta,s}}.
\]
Denote by $\mathrm{N}_{q^n/ q^{\frac{n}{2}}}:\fqn\to\mathbb{F}_{q^{\frac{n}{2}}}$ the norm function $x\mapsto x^{1+q^{n/2}}$ of $\fqn$ over $\mathbb{F}_{q^{\frac{n}{2}}}$.
\begin{itemize}
    \item If $\mathrm{N}_{q^n/ q^{\frac{n}{2}}}(\delta)=\mathrm{N}_{q^n/ q^{\frac{n}{2}}}(\delta^\prime)$, then $\mathcal{C}_{\delta,s}$ is MRD if and only if $\mathcal{C}_{\delta^\prime,s}$ is MRD; see \cite[Section 5]{CMPZ}.
    \item If $\mathrm{N}_{q^n/ q^{\frac{n}{2}}}(\delta)=1$ and $n>2$, then $\dim_{\fq}\ker(f_{\delta,s}(x))>1$ and  $\mathcal{C}_{\delta,s}$ is not MRD.
    \item If $\mathrm{N}_{q^n/ q^{\frac{n}{2}}}(\delta)\ne1$, then the rank of any non-zero element of $\mathcal{C}_{\delta,s}$ is at least $n-2$, and hence $\mathcal{C}_{\delta,s}$ is either Almost MRD or MRD; see \cite[Proposition 4.1]{CMPZ}.
    \item If $n=2$, then $\mathcal{C}_{\delta,s}=\mathcal{L}_{2,q}$ is trivially MRD.
    \item If $n=4$, then $\mathcal{C}_{\delta,s}$ is equivalent to a so-called twisted Gabidulin code, and is MRD if and only if $\delta^{1+q+q^2+q^3}\ne1$; see \cite{John}.
    \item If $n=6$, then there are exactly $\lceil(q^2+q+1)(q-2)\rceil$ values of $\mathrm{N}_{q^6/q^3}(\delta)$ for which $\mathcal{C}_{\delta,s}$ is MRD, and a characterization of such values of $\mathrm{N}_{q^6/q^3}(\delta)$ is known; see \cite[Theorem 7.3]{PZnumber} and \cite{BCM}.
    \item If $n=8$, $q$ is odd and $\mathrm{N}_{q^8/q^4}(\delta)=-1$, then $\mathcal{C}_{\delta,s}$ is MRD; see \cite[Theorem 7.2]{CMPZ}.
    \item If $n=8$, $q$ is odd and $q\leq11$, then $\mathcal{C}_{\delta,s}$ is MRD if and only if $\mathrm{N}_{q^8/q^4}(\delta)=-1$; see \cite[Remark 7.4]{CMPZ}.
    \item If $n\geq10$ and
    \[
    n \geq\begin{cases}
            8s+4 & \mbox{if }q=3\mbox{ and }s>1,\mbox{ or }q=2\mbox{ and }s>2, \\
            8s+2 & \mbox{otherwise},
    \end{cases}
    \]
    then $\mathcal{C}_{\delta,s}$ is not MRD; see \cite[Theorem 4.5]{PZZ}.
\end{itemize}

\begin{remark}\label{rem:pseudo}
We have defined the codes $\mathcal{C}_{\delta,s}$ only when $\delta\ne0$, which is the case under investigation in this paper. For completeness we mention here what happens when $\delta=0$: the code $\langle x,x^{q^s}\rangle_{\mathbb{F}_{q^n}}\subseteq\mathcal{L}_{n,q}$ is MRD if and only if $s$ is coprime with $n$. In this case, $\langle x,x^{q^s}\rangle_{\mathbb{F}_{q^n}}$ is a so-called generalized Gabidulin code; see \cite{gab2}.
\end{remark}

Recall that two subsets $S_1,S_2$ of $\mathbb{F}_{q^n}\times\mathbb{F}_{q^n}$ are \emph{${\rm \Gamma L}(2,q^n)$-equivalent} if $S_2=\Sigma(S_1)$ for some invertible semilinear map $\Sigma\in{\rm \Gamma L}(2,q^n)$, that is, $\Sigma=L\circ\varphi$ where $L\in{\rm GL}(2,q^n)$ and $\varphi$ acts elementwise as a field automorphism of $\mathbb{F}_{q^n}$.

For MRD codes, equivalence of codes $\mathcal{C}_f$ corresponds to ${\rm \Gamma L}$-equivalence of subspaces $U_f$.

\begin{theorem}\cite[Theorem 8]{John}\label{JohnEquiv}
Let $f(x),g(x)\in\mathcal{L}_{n,q}$ be such that $\mathcal{C}_{f},\mathcal{C}_{g}\subseteq\mathcal{L}_{n,q}$ are MRD codes. Then the $\fq$-vector subspaces $U_f$ and $U_g$ of $\mathbb{F}_{q^n}\times\mathbb{F}_{q^n}$ are ${\rm \Gamma L}(2,q^n)$-equivalent if and only if $\mathcal{C}_f$ and $\mathcal{C}_g$ are equivalent.
\end{theorem}

The ${\rm \Gamma L}$-equivalence between subspaces $U_{\delta,s}$ has been determined in \cite{CMPZ}.

\begin{proposition}\cite[Proposition 5.1]{CMPZ}\label{fequiv}
Let $n\geq4$ be even, $1\leq s,s^\prime< \frac{n}{2}$ be such that $\gcd(s,\frac{n}{2})=\gcd(s^\prime,\frac{n}{2})=1$, and $\delta,\delta^\prime\in\mathbb{F}_{q^{n}}^*$ satisfy $\mathrm{N}_{q^n/q^{\frac{n}{2}}}(\delta)\ne1$, $\mathrm{N}_{q^n/q^{\frac{n}{2}}}(\delta^\prime)\ne1$.
Then $U_{\delta,s}$ and $U_{\delta^\prime,s^\prime}$ are ${\rm \Gamma L}(2,q^{n})$-equivalent if and only if one of the following cases occurs for some automorphism $\sigma\in{\rm Aut}(\mathbb{F}_{q^{n/2}})$:
\begin{itemize}
    \item $s'=s$ and $\mathrm{N}_{q^n/q^{\frac{n}{2}}}(\delta)=(\mathrm{N}_{q^n/q^{\frac{n}{2}}}(\delta^\prime))^\sigma$;
    \item $s'+s=n/2$ and  $\mathrm{N}_{q^n/q^{\frac{n}{2}}}(\delta)\cdot(\mathrm{N}_{q^n/q^{\frac{n}{2}}}(\delta^\prime))^\sigma=1$.
\end{itemize}
\end{proposition}

\subsection{An algebraic curve attached to $\mathcal{C}_{\delta,s}$}\label{sec:curve}

We report here a characterization of MRD codes $\mathcal{C}_{\delta,s}$ provided in \cite{PZZ}, which is our starting point for Section \ref{sec:MRD}.
The discussion in \cite{PZZ} is in terms of kernels of the $q$-polynomials in $\mathcal{C}_{\delta,s}$; we report it in terms of the MRD property for $\mathcal{C}_{\delta,s}$, by means of Remark \ref{rem:corrisp}.

Theorem \ref{th:charact} is a sufficient condition for $\mathcal{C}_{\delta,s}$ not being MRD.
\begin{theorem}\label{th:charact}{\rm \cite[Theorem 3.6]{PZZ}}
Let $\delta\in\mathbb{F}_{q^n}^*$ be such that $\alpha:=\mathrm{N}_{q^n/q^{\frac{n}{2}}}(\delta)\in\mathbb{F}_{q^{n/2}}^*$ satisfies $\alpha\ne1$, and $s\in\{1,\ldots,n-1\}$ be coprime with $n/2$.
If there exist $T,S,A,B\in\mathbb{F}_{q^{n/2}}$ such that
\begin{enumerate}
    \item[(i)] $(1-\alpha)(T+T^{q^s})-\alpha S^{q^s+1}+(1+\alpha)(AS-2BT)=0$,
    \item[(ii)] $X^2-SX-T\in\mathbb{F}_{q^{n/2}}[X]$ is irreducible over $\mathbb{F}_{q^{n/2}}$,
    \item[(iii)] $S^{q^s}=2A+BS$,
    \item[(iv)] $-T^{q^s}=A^2+B(AS-BT)$,
\end{enumerate}
then $\mathcal{C}_{\delta,s}\subseteq\mathcal{L}_{n,q}$ is not MRD.
\end{theorem}

We will now translate the existence of $T,S,A,B$ satisfying the assumptions of Theorem \ref{th:charact} into the existence of a suitable point of an $\mathbb{F}_{q^{n/2}}$-rational algebraic curve. The equation of this curve is obtained through arithmetic manipulations which depend on $q$ being even or odd.
We will give the resulting equation only in the case when $q$ is odd, because this is the case that will be developed in the next sections. 
Therefore, assume in the rest of the section that $q$ is odd.

Let $T,S,A,B$ be as in the assumptions of Theorem \ref{th:charact}. From Conditions (iii) and (iv) it follows that
\[
B=\epsilon\Delta^{\frac{q^s-1}{2}}, \quad A=\frac{1}{2}(S^{q^s}-\epsilon S\Delta^{\frac{q^s-1}{2}}),
\]
where $\Delta=S^2+4T$ and $\epsilon$ is either $1$ or $-1$.
Write $\beta:=\epsilon\frac{\alpha+1}{1-\alpha}$,
and choose any non-square element $\eta$ of $\mathbb{F}_{q^{n/2}}$. The irreducibility condition (ii) in Theorem \ref{th:charact} is equivalent to the existence of an element $Z\in\mathbb{F}_{q^{n/2}}^*$ such that $\Delta=\eta Z^2$. Now, using $T=\frac{\eta Z^2-S^2}{4}$ together with Condition (iv), we obtain the following equation from Condition (i):
\begin{equation}\label{eq:eccola}
(S^{q^s}-S)^2 = \eta Z^2 + \eta^{q^s} Z^{2q^s} - 2\beta\eta^{\frac{q^s+1}{2}}Z^{q^s+1}.
\end{equation}
Viceversa, suppose that $\beta=\epsilon\frac{\alpha+1}{1-\alpha}$ with $\epsilon\in\{1,-1\}$, and that Equation \eqref{eq:eccola} is satisfied for some $S,Z\in\mathbb{F}_{q^{n/2}}$ with $Z\ne0$ and some non-square $\eta$ of $\mathbb{F}_{q^{n/2}}$. Then clearly there exist $T,S,A,B$ satisfying the assumptions of Theorem \ref{th:charact}.

In terms of the algebraic plane curve $\mathcal{X}_{\delta,s}$ with affine equation
\begin{equation}\label{eq:curveqodd_withs}
\mathcal{X}_{\delta,s}:\quad -(S^{q^s}-S)^2+\eta Z^2 + \eta^{q^s} Z^{2q^s} - 2\beta\eta^{\frac{q^s+1}{2}}Z^{q^s+1}=0,
\end{equation}
the discussion above proves the following proposition.

\begin{proposition}{\rm \cite[Section 3.1]{PZZ}}\label{prop:condX}
Let $\eta$ be a non-square in $\mathbb{F}_{q^{n/2}}$,
Let $\delta\in\mathbb{F}_{q^n}^*$, $s$ be coprime with $n/2$, $\alpha=\mathrm{N}_{q^n/q^{\frac{n}{2}}}(\delta)$ with $\alpha\ne1$, $\epsilon\in\{1,-1\}$ and $\beta:=\epsilon\frac{\alpha+1}{1-\alpha}$.
If for some non-square $\eta$ of $\mathbb{F}_{q^{n/2}}$ the curve $\mathcal{X}_{\delta,s}$ has an $\mathbb{F}_{q^{n/2}}$-rational affine point $(\bar{s},\bar{z})$ with $\bar{z}\ne0$, then the code $\mathcal{C}_{\delta,s}\subseteq\mathcal{L}_{n,q}$ is not MRD.
\end{proposition}

With the same notation as above, notice that $\beta\ne1$ and $\beta\ne-1$, because $\delta\ne0$; see Remark \ref{rem:pseudo}.
Moreover, if $\beta=0$ then $\mathrm{N}_{q^n/q^{\frac{n}{2}}}(\delta)=-1$; when $n=8$, this condition on $\delta$ yields MRD codes $\mathcal{C}_{\delta,s}$, see Section \ref{sec:ourcodes}.
Therefore, for $n=8$ we can assume in our investigation that $\beta\notin\{0,1,-1\}$, and the following theorem follows from Proposition \ref{prop:condX}.

\begin{theorem}\label{th:anybeta}
Let $q$ be an odd prime power and $s\in\{1,3,5,7\}$.
Suppose that for any $\beta\in\mathbb{F}_{q^4}\setminus\{0,1,-1\}$ there exists a non-square $\eta$ of $\mathbb{F}_{q^4}$ such that the curve $\mathcal{X}_{\delta,s}$ with equation \eqref{eq:curveqodd_withs} has an $\mathbb{F}_{q^4}$-rational affine point $(\bar{s},\bar{z})$ with $\bar{z}\ne0$.
Then $\mathcal{C}_{\delta,s}$ is MRD if and only if $\delta\in\mathbb{F}_{q^8}^*$ satisfies $\mathrm{N}_{q^8/q^4}(\delta)=-1$.
\end{theorem}

Theorem \ref{th:anybeta} will be the key tool for the proof of the first part of Theorem \ref{th:main} in the next section.
For the sake of completeness, we now describe how Proposition \ref{prop:condX} was used in \cite{PZZ} to study the MRD property for $\mathcal{C}_{\delta,s}$ for larger values of $n$. 

If $\beta\notin\{1,-1\}$, then the curve $\mathcal{X}_{\delta,s}$ is absolutely irreducible, and has genus $g(\mathcal{X}_{\delta,s})=q^{2s}-q^s-1$; see \cite[Theorem 3.7]{PZZ}.
Thus, the Hasse-Weil lower bound
\begin{equation}\label{eq:N}
N_{q^{n/2}} \geq q^{n/2}+1-2(q^{2s}-q^s-1)\sqrt{q^{n/2}}
\end{equation}
holds for the number $N_{q^{n/2}}$ of rational places of the $\mathbb{F}_{q^{n/2}}$-rational curve $\mathcal{X}_{\delta,s}$. If
\begin{equation}\label{eq:condPZZ}
n\geq\begin{cases}
        8s+4 & \mbox{if }q=3\mbox{ and }s>1,\\
        8s+2 & \mbox{otherwise},
\end{cases}
\end{equation}
then the condition \eqref{eq:N} implies that $N_{q^{n/2}}$ is positive and large enough, so that $\mathcal{X}_{\delta,s}$ has an $\mathbb{F}_{q^{n/2}}$-rational affine point $(\bar{s},\bar{z})$ with $\bar{z}\ne0$; see \cite[Proposition 3.8]{PZZ}.
By Proposition \ref{prop:condX}, this shows for any $\delta\in\mathbb{F}_{q^n}^*$ that under the condition \eqref{eq:condPZZ} the code $\mathcal{C}_{\delta,s}$ is not MRD; this is the statement of \cite[Theorem 1.1]{PZZ} for $q$ odd.

In this way the question on the MRD property for $\mathcal{C}_{\delta,s}$ is completely answered when $n$ is large enough with respect to $s$; for instance, when $s=1$ and $n\geq10$.
On the contrary, the right-hand side of \eqref{eq:condPZZ} is negative when $n=8$, and a different approach to the curve $\mathcal{X}_{\delta,s}$ is required. This is the object of Section \ref{sec:MRD}.

\section{Characterization of MRD codes $\mathcal{C}_{\delta,s}\subseteq\mathcal{L}_{8,q}$}\label{sec:MRD}

In this section we prove the first part of Theorem \ref{th:main}, namely Proposition \ref{prop:anys}.
Through the whole section the prime power $q$ is always assumed to be odd, so that we can consider the curve $\mathcal{X}_{\delta,s}$ in \eqref{eq:curveqodd_withs} and try to apply Theorem \ref{th:anybeta}.

We start by showing that we can restrict to the case $s=1$.

\begin{proposition}\label{prop:from1toanys}
Let $\delta\in\mathbb{F}_{q^8}^*$ with $\mathrm{N}_{q^8/q^4}(\delta)\ne1$. For any $s\in\{1,3,5,7\}$, the code $\mathcal{C}_{\delta,s}$ is equivalent to $\mathcal{C}_{\tilde{\delta},1}$ for some $\tilde{\delta}\in\mathbb{F}_{q^8}^*$ such that $\mathrm{N}_{q^8/q^4}(\tilde{\delta})=-1$ if and only if $\mathrm{N}_{q^8/q^4}(\delta)=-1$.
\end{proposition}

\begin{proof}
For $s=1$ or $s=3$, the claim follows from Theorem \ref{JohnEquiv} and Proposition \ref{fequiv}, as $\sigma(-1)=-1$ for any $\sigma\in{\rm Aut}(\mathbb{F}_{q^4})$.

For $s=5$ we have $f_{\delta,5}(x)=\delta f_{1/\delta,1}(x)$ and hence $\mathcal{C}_{\delta,5}=\mathcal{C}_{1/\delta,1}$, while for $s=7$ we have $f_{\delta,7}(x)=\delta f_{1/\delta,3}(x)$ and hence $\mathcal{C}_{\delta,7}=\mathcal{C}_{1/\delta,3}$.
Since ${\rm N}_{q^8/q^{4}}(\delta)=-1$ if and only if ${\rm N}_{q^8/q^{4}}(1/\delta)=-1$, and the claim holds for $s=1$ and $s=3$, it follows that the claim holds also for $s=5$ and $s=7$.
\end{proof}

We can then assume from now on that $s=1$, so that the curve to be considered has equation
\begin{equation}\label{eq:curveqodd}
\mathcal{X}_{\delta,1}:\quad -(S^{q}-S)^2+\eta Z^2 + \eta^{q} Z^{2q} - 2\beta\eta^{\frac{q+1}{2}}Z^{q+1}=0,
\end{equation}
where $\beta\in\mathbb{F}_{q^4}\setminus\{0,1,-1\}$ and $\eta$ is a non-square in $\mathbb{F}_{q^4}$.


Let $\xi$ be a normal element of $\fqq$ over $\fq$, and write
\begin{equation}\label{eq:espressioni}
S=S_0\xi+S_1\xi^q+S_2\xi^{q^2}+S_3\xi^{q^3},\qquad Z=Z_0\xi+Z_1\xi^q+Z_2\xi^{q^2}+Z_3\xi^{q^3},
\end{equation}
Since $\mathcal{B}=\{\xi,\xi^q,\xi^{q^2},\xi^{q^3}\}$ is an $\mathbb{F}_q$-basis of $\mathbb{F}_{q^4}$, Equation \ref{eq:espressioni} gives a one-to-one correspondence $(S,Z)\mapsto(S_0,S_1,S_2,S_3,Z_0,Z_1,Z_2,Z_3)$ between $\mathbb{F}_{q^4}^2$ and $\mathbb{F}_q^8$.

Now plug $S,Z\in\mathbb{F}_{q^4}$ from Equation \eqref{eq:espressioni} into \eqref{eq:curveqodd} and write it as a zero $\fq$-linear combination
\[
f_0 \xi + f_1 \xi^q + f_2 \xi^{q^2} + f_3 \xi^{q^3} = 0
\]
in the basis $\mathcal{B}$. Since the left-hand side of \eqref{eq:curveqodd} is a quadratic form over $\fq$, the coefficients $f_i$ are homogeneous quadratic polynomials over $\fq$ in the $8$ indeterminates $S_i,Z_i$. Thus, the equations
\[
\mathcal{W}\colon\begin{cases}
f_0(S_0,S_1,S_2,S_3,Z_0,Z_1,Z_2,Z_3)=0,\\
f_1(S_0,S_1,S_2,S_3,Z_0,Z_1,Z_2,Z_3)=0,\\
f_2(S_0,S_1,S_2,S_3,Z_0,Z_1,Z_2,Z_3)=0,\\
f_3(S_0,S_1,S_2,S_3,Z_0,Z_1,Z_2,Z_3)=0\\
\end{cases}
\]
define an $\fq$-rational projective variety $\mathcal{W}$ in $\mathbb{P}^7$. We have shown that the $\fq$-rational points of $\mathcal{W}$ give information on the $\mathbb{F}_{q^4}$-rational points of $\mathcal{X}_{\delta,1}$, as follows.
\begin{lemma}\label{lem:W}
If $\mathcal{W}$ has an $\fq$-rational point $(S_0:S_1:S_2:S_3:Z_0:Z_1:Z_2:Z_3)$ with $Z_3\ne0$, then $\mathcal{X}_{\delta,1}$ has an $\mathbb{F}_{q^4}$-rational affine point $(S,Z)$ with $Z\ne0$, given by Equation \eqref{eq:espressioni}.
\end{lemma}
By Theorem \ref{th:anybeta} and Lemma \ref{lem:W}, it is enough to show that $\mathcal{W}$ has an $\fq$-rational point for any $\beta\notin\{0,1,-1\}$ and for some non-square $\eta$ of $\mathbb{F}_{q^4}$. To do this, we will prove that $\mathcal{W}$ has an $\fq$-rational absolutely irreducible component. To this aim, we will study the absolutely irreducible components of another variety $\mathcal{V}$ which is projectively equivalent to $\mathcal{W}$.

Consider the Moore matrix
\[ M=
\begin{pmatrix}
\xi & \xi^q & \xi^{q^2} & \xi^{q^3} \\
\xi^{q} & \xi^{q^2} & \xi^{q^3} & \xi \\
\xi^{q^2} & \xi^{q^3} & \xi & \xi^{q} \\
\xi^{q^3} & \xi & \xi^q & \xi^{q^2} \\
\end{pmatrix}
\]
over $\mathbb{F}_{q^4}$. Since $\mathcal{B}$ is an $\fq$-basis of $\fqq$, we have $\det(M)\ne0$; see \cite{Moore}.
Thus the map $\varphi$ defined by
\[
(S_0:S_1:S_2:S_3:Z_0:Z_1:Z_2:Z_3)\mapsto(X_0:X_1:X_2:X_3:Y_0:Y_1:Y_2:Y_3):= (S_0,S_1,S_2,S_3,Z_0,Z_1,Z_2,Z_3)\cdot
\begin{pmatrix}
M & 0 \\ 0 & M \\
\end{pmatrix}
\]
is an $\fqq$-rational projectivity of $\mathbb{P}^7$.

Whenever the coordinates $S_i,Z_i$ of a point $P\in\mathbb{P}^7$ are in $\fq$ and $S,Z$ are defined as in \eqref{eq:espressioni}, the coordinates $X_i,Y_i$ of $\varphi(P)$ satisfy $X_i=S^{q^i}$ and $Y_i=Z^{q^i}$.
Therefore, the equations defining the image $\varphi(\mathcal{W})$ are obtained by applying the $q^j$-power to Equation \eqref{eq:curveqodd} with $j=0,\ldots,3$, and replacing $S^{q^i},Z^{q^i}$ with $X_i,Y_i$.
One gets
\[ \varphi(\mathcal{W}):\begin{cases}
(X_1-X_0)^2 = \eta Y_0^2 + \eta^{q} Y_1^2 - 2\beta\eta^{\frac{q+1}{2}}Y_0 Y_1,\\
(X_2-X_1)^2 = \eta^q Y_1^2 + \eta^{q^2} Y_2^2 - 2\beta^q\eta^{\frac{q^2+q}{2}}Y_1 Y_2,\\
(X_3-X_2)^2 = \eta^{q^2} Y_2^2 + \eta^{q^3} Y_3^2 - 2\beta^{q^2}\eta^{\frac{q^3+q^2}{2}}Y_2 Y_3,\\
(X_0-X_3)^2 = \eta^{q^3} Y_3^2 + \eta Y_0^2 - 2\beta^{q^3}\eta^{\frac{1+q^3}{2}}Y_3 Y_0.\\
\end{cases} \]
Clearly $\varphi(\mathcal{W})$ is defined over $\fqq$, since $\varphi$ is $\fqq$-rational and $\mathcal{W}$ is defined over $\fq$.

Let $\psi:\mathbb{P}^7\to\mathbb{P}^7$ be the $\fq$-rational projectivity defined by
\[
(X_0:X_1:X_2:X_3:Y_0:Y_1:Y_2:Y_3)\mapsto (W_0=X_1-X_0:W_1=X_2-X_1:W_2=X_3-X_2:X_3:Y_0:Y_1:Y_2:Y_3).
\] 
Then the variety $\mathcal{V}:=\psi(\varphi(\mathcal{W}))$ is defined over $\fqq$ and has equations
\begin{equation}\label{eq:V}
\mathcal{V}:\begin{cases}
W_0^2 = \eta Y_0^2 + \eta^{q} Y_1^2 - 2\beta\eta^{\frac{q+1}{2}}Y_0 Y_1,\\
W_1^2 = \eta^q Y_1^2 + \eta^{q^2} Y_2^2 - 2\beta^q\eta^{\frac{q^2+q}{2}}Y_1 Y_2,\\
W_2^2 = \eta^{q^2} Y_2^2 + \eta^{q^3} Y_3^2 - 2\beta^{q^2}\eta^{\frac{q^3+q^2}{2}}Y_2 Y_3,\\
(W_0+W_1+W_2)^2 = \eta^{q^3} Y_3^2 + \eta Y_0^2 - 2\beta^{q^3}\eta^{\frac{1+q^3}{2}}Y_3 Y_0.\\
\end{cases}
\end{equation}

Note that $\mathcal{V}$ is a cone with vertex the point $(0:0:0:1:0:0:0:0)$.

We will use the following lemma to get information on $\mathcal{W}$ from the components of $\mathcal{V}$.

\begin{lemma}\label{lemma:fromVtoW}
Suppose that $\mathcal{V}$ has an absolutely irreducible component $\mathcal{V}^{\prime}$ of dimension $m$ and degree $d$, such that no other absolutely irreducible component of $\mathcal{V}$ has dimension $m$ and degree $d$. Then $\mathcal{W}$ has an absolutely irreducible component $\mathcal{W}^{\prime}$ which is defined over $\fq$ and has dimension $m$ and degree $d$.
\end{lemma}

\begin{proof}
Recall that projectivities between varieties preserve absolute irreducibility, dimension and degree of the components.
Since $\mathcal{W}$ is projectively equivalent to $\mathcal{V}$ through the projectivity $\pi:=(\psi\circ\varphi)^{-1}$, the variety $\mathcal{W}^{\prime}:=\pi(\mathcal{V}^{\prime})$ is the only absolutely irreducible component of $\mathcal{W}$ with dimension $m$ and degree $d$.

Let $\Phi_q:P\mapsto P^q$ be the $q$-Frobenius map, which raises the coordinates of a point to power $q$.
Since $\mathcal{W}$ is defined over $\fq$, $\Phi_q$ preserves $\mathcal{W}$. Being a collineation, $\Phi_q$ maps absolutely irreducible components of $\mathcal{W}$ to absolutely irreducible components of $\mathcal{W}$ with the same dimension and degree.
Since $\mathcal{W}^{\prime}$ is the only absolutely irreducible component of $\mathcal{W}$ with dimension $d$ and degree $m$, $\Phi_q$ preserves $\mathcal{W}^{\prime}$. Therefore, $\mathcal{W}^{\prime}$ is defined over $\fq$.
\end{proof}

The dimension and degree of $\mathcal{V}$ are given in Lemma \ref{lemma:V}, whose proof is postponed to Section \ref{sec:V}.

\begin{lemma}\label{lemma:V}
The variety $\mathcal{V}\subset\mathbb{P}^7$ has dimension $3$ and degree $16$.
\end{lemma}

We first consider the case $\beta^{2q}\ne\beta^2$, starting by the intersection of $\mathcal{V}$ with a suitable hyperplane. The proof of Lemma \ref{lemma:irreduciblecurve} is postponed to Section \ref{sec:lemmairred}.

\begin{lemma}\label{lemma:irreduciblecurve}
Let $\beta\in\fqq$ be such that $\beta^{2q}\ne\beta^{2}$.
For some non-square $\eta$ of $\fqq$ and infinitely many elements $k\in\KK$, the intersection between $\mathcal{V}$ and the hyperplane $\Pi$ with affine equation $Y_2=Y_1+k$ is a non-repeated absolutely irreducible surface $\mathcal{S}$.
\end{lemma}

Lemma \ref{lemma:irreduciblecurve} allows us to find in Proposition \ref{prop:componenti} a suitable component of $\mathcal{V}$.

\begin{proposition}\label{prop:componenti}
Let $\beta\in\mathbb{F}_{q^4}$ be such that $\beta^{2q}\ne\beta^2$.
For some non-square $\eta$ of $\fqq$, $\mathcal{V}$ contains a unique absolutely irreducible component $\mathcal{V}^{\prime}$ of dimension $3$ and degree $16$.
\end{proposition}

\begin{proof}
Let $\eta$, $k$ and $\Pi$ be as in the claim of Lemma \ref{lemma:irreduciblecurve}, so that $\mathcal{S}:=\mathcal{V}\cap\Pi$ is a non-repeated absolutely irreducible surface.
Suppose by contradiction that $\mathcal{V}$ has more than one absolutely irreducible component of dimension $3$, and let $\mathcal{V}_1$ and $\mathcal{V}_2$ be two of them. By Proposition \ref{prop:sha}(iii), both $\Pi\cap\mathcal{V}_1$ and $\Pi\cap\mathcal{V}_2$ have dimension at least $2$,
and hence $\mathcal{S}=\Pi\cap\mathcal{V}\supseteq (\Pi\cap \mathcal{V}_1) \cup (\Pi\cap \mathcal{V}_2)$ contains two distinct components of dimension $2$, or a repeated component of dimension $2$. This is a contradiction to $\mathcal{S}$ being non-repeated and absolutely irreducible.
Therefore $\mathcal{V}$ has a unique absolutely irreducible component $\mathcal{V}^{\prime}$ of dimension $3$. Being the unique component of maximal dimension, $\mathcal{V}^{\prime}$ has the same degree $16$ as $\mathcal{V}$.
\end{proof}

The remaining cases for $\beta$ are considered in Lemma \ref{lemma:reduciblecurve}, whose proof is postponed to Section \ref{sec:lemmaqodd2}.

\begin{lemma}\label{lemma:reduciblecurve}
Let $\beta\in\mathbb{F}_{q^4}\setminus\{0,1,-1\}$ be such that $\beta^{2q}=\beta^2$.
Then, for some non-square $\eta$ of $\mathbb{F}_{q^4}$, $\mathcal{V}$ has exactly $3$ irreducible components of dimension $3$, and only one of them, say $\mathcal{V}^{\prime}$, has degree $8$.
\end{lemma}

We can now show that $\mathcal{W}$ has $\fq$-rational points when $q$ is big enough.

\begin{theorem}\label{th:geomfin}
Let $q\geq 1039891$ be an odd prime power, and $\beta\in\fqq\setminus\{0,1,-1\}$. Then, for some non-square $\eta$ of $\fqq$, $\mathcal{W}$ has an $\fq$-rational point.
\end{theorem}

\begin{proof}
If $\beta^{2q}\ne\beta^2$ then $\mathcal{W}$ has an $\fq$-rational absolutely irreducible $3$-dimensional component $\mathcal{W}^{\prime}$ of degree $16$, by Proposition \ref{prop:componenti} and Lemma \ref{lemma:fromVtoW}.
If $\beta^{2q}=\beta^2$ then $\mathcal{W}$ has an $\fq$-rational absolutely irreducible $3$-dimensional component $\mathcal{W}^{\prime}$ of degree $8$, by Lemma \ref{lemma:reduciblecurve} and Lemma \ref{lemma:fromVtoW}.
In both cases we can apply Theorem \ref{Th:CafureMatera} to $\mathcal{W}^{\prime}$, with $m=3$ and $d\leq16$. Here, we choose the affine points of $\mathcal{W}^{\prime}$ by dehomogenizing with respect to a $Z_i$-coordinate, say $Z_3$.
Whenever $q\geq 2\cdot(3+1)\cdot16^2=2048$, this yields
\[
A_q\geq q^3 - 15\cdot14\cdot q^{5/2} - 5\cdot 16^{13/3}\cdot q^2.
\]
For $q\geq 1039891$, this implies $A_q>0$. Therefore $\mathcal{W}^{\prime}$ has an $\fq$-rational point with $Z_3\ne0$, and the claim is proved. 
\end{proof}

By Theorem \ref{th:geomfin} and Lemma \ref{lem:W}, the hypothesis of Theorem \ref{th:anybeta} holds when $q$ is big enough. Therefore, by Theorem \ref{th:anybeta}, the following is proved.

\begin{proposition}\label{prop:s=1}
If $q\geq1039891$ is an odd prime power and $\delta\in\mathbb{F}_{q^8}^*$, then $\mathcal{C}_{\delta,1}\subseteq\mathcal{L}_{8,q}$ is MRD if and only if $\mathrm{N}_{q^8/q^4}(\delta)=-1$.
\end{proposition}

By Proposition \ref{prop:from1toanys}, the claim of Proposition \ref{prop:s=1} holds also for any $s\in\{1,3,5,7\}$.

\begin{proposition}\label{prop:anys}
If $q\geq1039891$ is an odd prime power, $\delta\in\mathbb{F}_{q^8}^*$ and $s\in\{1,3,5,7\}$, then $\mathcal{C}_{\delta,s}\subseteq\mathcal{L}_{8,q}$ is MRD if and only if $\mathrm{N}_{q^8/q^4}(\delta)=-1$.
\end{proposition}

\section{Parameters and equivalences for MRD codes $\mathcal{C}_{\delta,s}\subseteq\mathcal{L}_{8,q}$}\label{sec:parequiv}

In this section, we prove the second part of Theorem \ref{th:main}.
Thus, we assume through the whole section that the prime power $q$ is odd.
We start from the parameters of $\mathcal{C}_{\delta,s}$.

\begin{lemma}
For any $\delta\in\mathbb{F}_{q^8}$ with $\mathrm{N}_{q^8/q^4}(\delta)=-1$, $\mathcal{C}_{\delta,s}$ has parameters $[8\times8,16,7]_q$ and left idealiser $\{\lambda x\,\colon\,\lambda\in\mathbb{F}_{q^8}\}\cong\mathbb{F}_{q^8}$. 
\end{lemma}

\begin{proof}
The size $8\times8$ is clear from $\mathcal{C}_{\delta,s}\subseteq\mathcal{L}_{8,q}$.
The $\fq$-dimension of $\mathcal{C}_{\delta,s}$ is $16$ because $\mathcal{C}_{\delta,s}$ is generated over $\mathbb{F}_{q^8}$ by the independent elements $x$ and $f_{\delta,s}(x)$.
The minimum distance of $\mathcal{C}_{\delta,s}$ is $7$ since $\mathcal{C}_{\delta,s}$ is MRD with size $8\times8$ and $\fq$-dimension $16$.
The left idealiser $L(\mathcal{C}_{\delta,s})$ contains the finite field $\{\tau_{\lambda}(x)=\lambda x\,\colon\,\lambda\in\mathbb{F}_{q^8}\}\cong\mathbb{F}_{q^8}$ and cannot be larger by Proposition \ref{prop:idealisers}.
\end{proof}

The remaining parameter to be determined is the right idealiser.

\begin{proposition}
For any $\delta\in\mathbb{F}_{q^8}$ with $\mathrm{N}_{q^8/q^4}(\delta)=-1$, $\mathcal{C}_{\delta,s}$ has right idealiser $\{\lambda x\,\colon\,\lambda\in\mathbb{F}_{q^4}\}\cong\mathbb{F}_{q^4}$. 
\end{proposition}
\begin{proof}
For any $\mu\in\mathbb{F}_{q^4}$, the polynomial $\tau_{\mu}(x)=\mu x$ satisfies $x\circ\tau_\mu(x)=\tau_{\mu}\circ x$ and $f_{\delta,s}\circ\tau_{\mu}(x)=\tau_{\mu^{q^s}}\circ f_{\delta,s}(x)$. As $\tau_{\mu}(x),\tau_{\mu^{q^s}}(x)\in L(\mathcal{C}_{\delta,s})$, this implies $\{\tau_{\mu}(x)\,\colon\, \mu\in\mathbb{F}_{q^4}\}\subseteq R(\mathcal{C}_{\delta,s})$.
Then, by Proposition \ref{prop:idealisers}, either $R(\mathcal{C}_{\delta,s})=\{\tau_{\mu}(x)\,\colon\, \mu\in\mathbb{F}_{q^4}\}\cong\mathbb{F}_{q^4}$ and the claim is proved, or $R(\mathcal{C}_{\delta,s})\cong\mathbb{F}_{q^8}$.

Suppose that the latter case holds, so that both $L(\mathcal{C}_{\delta,s})$ and $R(\mathcal{C}_{\delta,s})$ are isomorphic to $\mathbb{F}_{q^8}$.
Then, by \cite[Theorem 2.2]{CMPZhou}, $\mathcal{C}_{\delta,s}$ is equivalent to $\langle x,x^{q^t}\rangle_{\mathbb{F}_{q^8}}$ for some $t$. As $\mathcal{C}_{\delta,s}$ is MRD, $t$ is coprime with $n$.
Therefore, by Theorem \ref{JohnEquiv}, $U_{\delta,s}$ is ${\rm \Gamma L}(2,q^8)$-equivalent to $U_{x^{q^t}}$.
This is a contradiction to \cite[Theorem 6.3]{CMPZ}.
\end{proof}

We now turn to the equivalence issue.
Up to our knowledge, every $\fq$-linear MRD code $\mathcal{C}$ in $\mathcal{L}_{8,q}$ with left idealiser isomorphic to $\mathbb{F}_{q^8}$ known so far in the literature is equivalent to one of the following.
\begin{itemize}
    \item[(i)] $\mathcal{C}_{\delta,s}=\langle x,x^{q^s}+\delta x^{q^{4+s}}\rangle_{\mathbb{F}_{q^8}}$ with $s\in\{1,3,5,7\}$ and $\mathrm{N}_{q^8/q^4}(\delta)=-1$.
    \item[(ii)] $\mathcal{C}_{G(r)}:=\langle x,x^{q^r}\rangle_{\mathbb{F}_{q^8}}$ with $r\in\{1,3,5,7\}$, so-called generalized Gabidulin codes; see \cite{gab2}.
    \item[(iii)] $\mathcal{C}_{T(\epsilon,r)}:=\langle x,\epsilon x^{q^r}+x^{q^{8-r}}\rangle_{\mathbb{F}_{q^8}}$ with $r\in\{1,3,5,7\}$ and $\mathrm{N}_{q^8/q}(\epsilon)\notin\{0,1\}$, so-called generalized twisted Gabidulin codes; see \cite[Remark 8]{John}.
    \item[(iv)] $\mathcal{C}_{Q(h,r)}:=\langle x, \psi_{h,r}(x) \rangle_{\mathbb{F}_{q^8}}$, where $\psi_{h,r}(x)$ is the quadrinomial $x^{q^r}+x^{q^{3r}} + h^{q^r+1}x^{q^{5r}} + h^{1-q^{7r}} x^{q^{7r}}$, with $r\in\{1,3,5,7\}$ and $\mathrm{N}_{q^8/q^4}(h)=-1$; see \cite{NSZ}.  
\end{itemize}

To complete the proof of Theorem \ref{th:main}, we show that $\mathcal{C}_{\delta,s}$ is not equivalent to any code in (ii), (iii) or (iv).

\begin{proposition}\label{prop:easy}
Let $s\in\{1,3,5,7\}$ and $\delta\in\mathbb{F}_{q^8}$ be such that $\mathrm{N}_{q^8/q^4}(\delta)=-1$.
For any $r\in\{1,3,5,7\}$ and $\epsilon\in\mathbb{F}_{q^8}$ such that $\mathrm{N}_{q^8/q}(\epsilon)\notin\{0,1\}$, $\mathcal{C}_{\delta,s}$ is not equivalent to $\mathcal{C}_{G(r)}$ nor to $\mathcal{C}_{T(\epsilon,r)}$.
\end{proposition}

\begin{proof}
By Theorem \ref{JohnEquiv}, the claim is equivalent to $U_{\delta,s}$ being non-${\rm \Gamma L}(2,q^8)$-equivalent to $U_{x^{q^r}}$ or $U_{\epsilon x^{q^r}+x^{q^{8-r}}}$.
This was already proved in \cite[Theorem 6.3]{CMPZ}.
\end{proof}

We now consider the codes $\mathcal{C}_{Q(h,r)}$ in (iv).
Let $t=n/2$. Theorem 4.6 in \cite{NSZ} provides an equivalence between $\mathcal{C}_{Q(h,r)}$ and $\mathcal{C}_{Q(\bar{h},1)}$ for some $\bar{h}\in\mathbb{F}_{q^{2t}}$ with $\mathrm{N}_{q^{2t}/q^t}(\bar{h})=-1$, under the assuption $t\geq5$.
Yet, the arguments used in the proof of \cite[Theorem 4.6]{NSZ} still hold when $t=4$ (see also \cite[Theorem 3.2 (ii)]{GLT}), hence proving Lemma \ref{lemma:restriz}.

\begin{lemma}{\rm (see the proof of \cite[Theorem 4.6]{NSZ})}\label{lemma:restriz}
Let $r\in\{1,3,5,7\}$ and $h\in\mathbb{F}_{q^8}$ with $\mathrm{N}_{q^8/q^4}(h)=-1$.
Then there exists $\bar{h}\in\mathbb{F}_{q^8}$ with $\mathrm{N}_{q^8/q^4}(\bar{h})=-1$ such that $\mathcal{C}_{Q(h,r)}$ is equivalent to $\mathcal{C}_{Q(\bar{h},1)}$.
\end{lemma}

Therefore we only have to consider the equivalence issue between $\mathcal{C}_{\delta,s}$ and $\mathcal{C}_{Q(h,1)}$.
We start from $s=1$.

\begin{proposition}\label{prop:nonequiv1}
For any $s\in\{1,3,5,7\}$ and $\delta,h\in\mathbb{F}_{q^8}$ with $\mathrm{N}_{q^8/q^4}(\delta)=\mathrm{N}_{q^8/q^4}(h)=-1$, the codes $\mathcal{C}_{\delta,s}$ and $\mathcal{C}_{Q(h,1)}$ are not equivalent.
\end{proposition}

\begin{proof}
By  Theorem \ref{JohnEquiv}, it is enough to prove that $U_{\delta,s}$ and $U_{\psi_{h,1}}$ are not ${\rm \Gamma L}(2,q^8)$-equivalent.
Assume by contradiction that there exist $\varphi\in{\rm Aut}(\mathbb{F}_{q^8})$ and $a,b,c,d\in\mathbb{F}_{q^8}$ such that $ad-bc\ne0$ and the following holds: for every $y\in\mathbb{F}_{q^8}$, there exists $z\in\mathbb{F}_{q^8}$ satisfying
\begin{equation}\label{eq:equiv}
\begin{pmatrix} a & b \\ c & d \\ \end{pmatrix}\cdot
\begin{pmatrix} y^{\varphi} \\ \psi_{h,1}(y)^{\varphi} \\ \end{pmatrix} =
\begin{pmatrix} z \\ f_{\delta,s}(z) \\ \end{pmatrix}.
\end{equation}
Write the action of $\varphi$ on $\mathbb{F}_{q^8}$ as $\lambda\mapsto(\lambda^{p^i})^{q^j}$, with $1\leq p^i<q$ and $0\leq j<8$. The map $(\mu,\rho)\mapsto(\mu^{q^j},\rho^{q^j})$ is $\fq$-linear on $\mathbb{F}_{q^8}\times\mathbb{F}_{q^8}$ and hence is in ${\rm GL}(2,q^8)$. Thus, up to changing $a,b,c,d$, we can assume that $j=0$.

The condition \eqref{eq:equiv} is equivalent to require that, for every $y\in\mathbb{F}_{q^8}$,
\[
c y^{\varphi} + d\psi_{h,1}(y)^{\varphi} - f_{\delta,s}(ay^{\varphi}+b\psi_{h,1}(y)^{\varphi}) = 0.
\]
Suppose $s=1$. Then every element of $\mathbb{F}_{q^8}$ is a root of the following polynomial:
\[
\begin{split}
P(x)&= \left(c-b^q h^{(q-1)p^i}-\delta b^{q^5}\right)x^{p^i} + \left(d-a^q\right)x^{qp^i}-\left(b^q+\delta b^{q^5} h^{(q^6+q^5)p^i}\right)x^{q^2p^i}+ d x^{q^3p^i}  \\
 & - \left(b^q+\delta b^{q^5}h^{(q^5-q^4)p^i}\right)x^{q^4p^i} + \left(d h^{(q+1)p^i}-\delta a^{q^5}\right) x^{q^5p^i} - \left(b^q h^{(q^2+q)p^i} + \delta b^{q^5}\right)x^{q^6p^i} + d h^{(1-q^7)p^i}x^{q^7p^i}.
\end{split}
\]
Since the degree of $P(x)$ is smaller than $q^8$, this implies that $P(x)$ is the zero polynomial.
The vanishing of its coefficients yields
\begin{equation}\label{eq:siste}
\begin{cases}
a=d=0, \\
c=b^q h^{(q-1)p^i}+\delta b^{q^5}, \\
b^q h^{(q^2+q)p^i} + \delta b^{q^5}=0, \\
b^q +\delta b^{q^5} h^{(q^6+q^5)p^i}=0, \\
b^q+\delta b^{q^5}h^{(q^5-q^4)p^i}=0. \\
\end{cases}
\end{equation}
Since $\delta\ne0$ and $b\ne0$ (as $ad-bc\ne0$), the last two equations in System \eqref{eq:siste} yield $h^{q^2+1}=1$. Since $4$ does not divide $q^2+1$, this yields a contradiction to $\mathrm{N}_{q^8/q^4}(h)=-1$.
This proves the claim for $s=1$, and hence for any $s\in\{1,3,5,7\}$ by Proposition \ref{prop:from1toanys}. 
\end{proof}

Therefore, for any $\delta\in\mathbb{F}_{q^8}^*$ with $\mathrm{N}_{q^8/q^4}=-1$ and any $s\in\{1,3,5,7\}$, the MRD code $\mathcal{C}_{\delta,s}\subseteq\mathcal{L}_{8,q}$ is not equivalent to a code $\mathcal{C}_{G(r)}$ or $\mathcal{C}_{T(\epsilon,r)}$ (Proposition \ref{prop:easy}), neither to a code $\mathcal{C}_{Q(h,r)}$ (Lemma \ref{lemma:restriz} and Proposition \ref{prop:nonequiv1}).
The proof of Theorem \ref{th:main} is now complete.

\section{Proof of Lemmas \ref{lemma:V}, \ref{lemma:irreduciblecurve}, \ref{lemma:reduciblecurve}}\label{sec:lemmaqodd}

In this section we prove three technical lemmas which were stated and applied in Section \ref{sec:MRD}.
The proofs of Lemmas \ref{lemma:irreduciblecurve} and \ref{lemma:reduciblecurve} use arguments from the theory of algebraic function fields and their finite extensions, for which we refer the reader to the monograph \cite{Sti}.
In particular, we will show that certain quadratic equations $X^2=u$ define Kummer extensions (see \cite[Proposition 3.7.3]{Sti}) of an algebraic function field $L$, by showing that $u$ is not a square in $L$.
To this aim, we will make use of Laurent series expansions with respect to local parameters at certain places (see \cite[Theorem 4.2.6]{Sti}).

\subsection{Proof of Lemma \ref{lemma:V}}\label{sec:V}

By Proposition \ref{prop:sha}(i), the dimension of $\mathcal{V}$ is at least $3$. To prove the equality, it is enough to provide a linear subspace $\mathcal{L}$ of $\mathbb{P}^7$ of dimension $3$ such that $\mathcal{V}\cap\mathcal{L}=\emptyset$; see Proposition \ref{prop:sha}(ii).
Let $\mathcal{L}$ be the $3$-dimensional linear space defined by the equations
\[
X_3=0,\quad Y_0=0,\quad Y_1=0,\quad Y_2+Y_3=0.
\]
If $\mathcal{V}\cap\mathcal{L}\ne\emptyset$, then any point $P=(\bar{w}_0:\bar{w}_1:\bar{w}_2:\bar{x}_3:\bar{y}_0:\bar{y}_1:\bar{y}_2:\bar{y}_3)$ in $\mathcal{V}\cap\mathcal{L}$ satisfies
\[
\bar{w}_0=\bar{x}_3=\bar{y}_0=\bar{y}_1=0,\quad \bar{w}_1^2=\eta^{q^2}\bar{y}_2^2,\quad \bar{w}_2^2 = (\eta^{q^2} + \eta^{q^3} + 2\beta^{q^2}\eta^{\frac{q^3+q^2}{2}})\bar{y}_2^2,\quad (\bar{w}_1+\bar{w}_2)^2 = \eta^{q^3} \bar{y}_2^2.
\]
This implies
\[
2(\eta^{q^2} + \beta^{q^2}\eta^{\frac{q^3+q^2}{2}})\bar{y}_2^2=-2\bar{w}_1\bar{w}_2,
\]
whence
\[
4(\eta^{q^2} + \beta^{q^2}\eta^{\frac{q^3+q^2}{2}})^2\bar{y}_2^4=4\bar{w}_1^2\bar{w}_2^2=4\cdot \eta^{q^2}\bar{y}_2^2\cdot (\eta^{q^2} + \eta^{q^3} + 2\beta^{q^2}\eta^{\frac{q^3+q^2}{2}})\bar{y}_2^2
\]
and thus
\[
4\eta^{q^3+q^2}(\beta^{2q^2}-1)\bar{y}_2^4=0.
\]
The case $\bar{y}_2=0$ cannot occur by the equations of $\mathcal{V}$ and $\mathcal{L}$, therefore $\beta^{2q^2}-1=0$, a contradiction to $\beta\ne\pm1$.

Finally, since $\mathcal{V}\subset\mathbb{P}^7$ has dimension $3$ and its ideal is generated by $4$ polynomials, $\mathcal{V}$ is a complete intersection. Thus, $\mathcal{V}$ has degree $2\cdot2\cdot2\cdot2=16$.

\subsection{Proof of Lemma \ref{lemma:irreduciblecurve}}\label{sec:lemmairred}

Note that the homogeneous equations of $\mathcal{V}$ can be dehomogenized with respect to the hyperplane at infinity $\mathcal{H}_0\colon Y_0=0$, since it is easily seen that $\mathcal{V}\cap\mathcal{H}_0$ has dimension strictly smaller than $3$.
Thus, given $k\in\KK=\overline{\mathbb{F}_{q^4}}$ and the hyperplane $\Pi:Y_2=Y_1+k$, $\mathcal{V}\cap\Pi$ has affine equations
\[
\mathcal{V}\cap\Pi: \begin{cases}
W_0^2 = \eta + \eta^{q} Y_1^2 - 2\beta\eta^{\frac{q+1}{2}}Y_1,\\
W_1^2 = \eta^q Y_1^2 + \eta^{q^2} (Y_1+k)^2 - 2\beta^q\eta^{\frac{q^2+q}{2}}Y_1(Y_1+k),\\
W_2^2 = \eta^{q^2} (Y_1+k)^2 + \eta^{q^3} Y_3^2 - 2\beta^{q^2}\eta^{\frac{q^3+q^2}{2}}(Y_1+k)Y_3,\\
(W_0+W_1+W_2)^2 = \eta^{q^3} Y_3^2 + \eta - 2\beta^{q^3}\eta^{\frac{1+q^3}{2}}Y_3\\
\end{cases}.
\]
Therefore, given a transcendental $y_1$ over $\KK$, if we prove that the equations
\begin{equation}\label{eq:lemmafunctionfield}
\begin{cases}
w_0^2 = \eta + \eta^{q} y_1^2 - 2\beta\eta^{\frac{q+1}{2}}y_1,\\
w_1^2 = \eta^q y_1^2 + \eta^{q^2} (y_1+k)^2 - 2\beta^q\eta^{\frac{q^2+q}{2}}y_1(y_1+k),\\
w_2^2 = \eta^{q^2} (y_1+k)^2 + \eta^{q^3} y_3^2 - 2\beta^{q^2}\eta^{\frac{q^3+q^2}{2}}(y_1+k)y_3,\\
(w_0+w_1+w_2)^2 = \eta^{q^3} y_3^2 + \eta - 2\beta^{q^3}\eta^{\frac{1+q^3}{2}}y_3\\
\end{cases}
\end{equation}
define an algebraic function field $\KK(y_1,y_3,w_0,w_1,w_2,x_3)/\KK$ of transcendence degree $2$, then $\KK(y_1,y_3,w_0,w_1,w_2,x_3)$ is the function field of a non-repeated absolutely irreducible surface $\mathcal{S}=\mathcal{V}\cap\Pi$ and Lemma \ref{lemma:irreduciblecurve} is proved.
Note that $x_3$ does not appear in Equations \eqref{eq:lemmafunctionfield}, and hence is transcendental over $\KK(y_1,y_3,w_0,w_1,w_2)$ (indeed, $\mathcal{V}\cap\Pi$ is a cone with vertex $(0:0:0:1:0:0:0:0)$). Thus, in order to prove Lemma \ref{lemma:irreduciblecurve} it is enough to show that Equations \eqref{eq:lemmafunctionfield} define an algebraic function field $\KK(y_1,y_3,w_0,w_1,w_2)/\KK$ of transcendence degree $1$.

After eliminating $y_3$ using the third and fourth equation of Equations \eqref{eq:lemmafunctionfield}, we obtain the following equivalent system:
\begin{equation}\label{eq:sist}
\begin{cases}
w_0^2 = \eta  + \eta^{q} y_1^2 - 2\beta\eta^{\frac{q+1}{2}}y_1,\\
w_1^2 = \eta^q y_1^2 + \eta^{q^2} (y_1+k)^2 - 2\beta^q\eta^{\frac{q^2+q}{2}}y_1 (y_1+k),\\
a w_2^2 +b w_2 + c = 0, \\
y_3 = \frac{ \eta^{q}y_1^2 + \eta^{q^2}(y_1+k)^2 - \beta\eta^{\frac{q+1}{2}}y_1 - \beta^q\eta^{\frac{q^2+q}{2}}y_1 (y_1+k) + w_0w_1 + w_1w_2 + w_0w_2 }{ \beta^{q^2}\eta^{\frac{q^3+q^2}{2}}(y_1+k) - \beta^{q^3}\eta^{\frac{q^3+1}{2}} }, \\
\end{cases}
\end{equation}
where
\[
\begin{aligned}
a \,=\, & 2w_0w_1 + \big(- 2\beta^q \eta^{\frac{q^2+q}{2}} - \beta^{2q^2}\eta^{q^2} + 2\eta^{q} + \eta^{q^2}\big)y_1^2 + 2\big( - \beta\eta^{\frac{q+1}{2}} - \beta^{q}\eta^{\frac{q^2+q}{2}}k -
        \beta^{2q^2}\eta^{q^2}k + \beta^{q^3+q^2}\eta^{\frac{q^2+1}{2}}\\
				& + \eta^{q^2}k\big)y_1 - \beta^{2q^2}\eta^{q^2}k^2 + 2\beta^{q^3+q^2}\eta^{\frac{q^2+1}{2}}k - \beta^{2q^3}\eta + \eta +
        \eta^{q^2}k^2, \\
b \,=\, & \beta^{q}\eta^{\frac{q^2+q}{2}}(-6w_0-2w_1)y_1^2 - 2\beta^{2q^2}\eta^{q^2}(w_0+w_1)y_1^2 + 4\eta^{q}(w_0+w_1)y_1^2 + \eta^{q^2}(4w_0+2w_1)y_1^2\\
& + \beta\eta^{\frac{q+1}{2}}(-2w_0-6w_1)y_1  + \beta^{q}\eta^{\frac{q^2+q}{2}}k (-6w_0-2w_1)y_1 -
    4\beta^{2q^2}\eta^{q^2}k (w_0+w_1)y_1 + 2\beta^{q^3+q^2}\eta^{\frac{q^2+1}{2}}(w_0+w_1)y_1\\
		&+ \eta^{q^2}k (8w_0+4w_1)y_1 - 2\beta^{2q^2}\eta^{q^2}k^2 (w_0+w_1) + 2\beta^{q^3+q^2}\eta^{\frac{q^2+1}{2}}k (w_0+w_1) + \eta^{q^2}k^2 (4w_0+2w_1) + 2\eta w_1,\\
c \,=\, & 2w_0w_1\big[ \big(-\beta^{q}\eta^{\frac{q^2+q}{2}} - \beta^{2q^2}\eta^{q^2} + \eta^{q} + \eta^{q^2} \big)y_1^2
+\big( -\beta\eta^{\frac{q+1}{2}} - \beta^{q}\eta^{\frac{q^2+q}{2}}k - 2\beta^{2q^2}\eta^{q^2}k + \beta^{q^3+q^2}\eta^{\frac{q^2+1}{2}} + 2\eta^{q^2}k \big)y_1 \\
& -\beta^{2q^2}\eta^{q^2}k^2 + \beta^{q^3+q^2}\eta^{\frac{q^2+1}{2}}k + \eta^{q^2}k^2 \big] + \big( \beta^{2q}\eta^{q^2+q} + 2\beta^{2q^2+q}\eta^{\frac{3q^2+q}{2}} - 4\beta^{q}\eta^{\frac{q^2+3q}{2}} - 2\beta^{q}\eta^{3q^2+q} - 2\beta^{2q^2}\eta^{q^2+q}\\
&  - \beta^{2q^2}\eta^{2q^2} + 2\eta^{2q} + 3\eta^{q^2+q} + \eta^{2q^2} \big)y_1^4 + 2\big(  3\beta^{q+1}\eta^{\frac{q^2+2q+1}{2}} + \beta^{2q^2+1}\eta^{\frac{2q^2+q+1}{2}} - 2\beta\eta^{\frac{3q+1}{2}} - 2\beta\eta^{\frac{2q^2+q+1}{2}} \\
& + \beta^{2q}\eta^{q^2+q}k + 3\beta^{2q^2+q}\eta^{\frac{3q^2}{2}}k - \beta^{q^3+q^2+q}\eta^{\frac{2q^2+q+1}{2}} - 2\beta^{q}\eta^{\frac{q^2+3q}{2}}k - 3\beta^{q}\eta^{\frac{3q^2+q}{2}}k - 2\beta^{2q^2}\eta^{q^2+q}k - 2\beta^{2q^2}\eta^{2q^2}k\\
& + \beta^{q^3+q^2}\eta^{\frac{q^2+2q+1}{2}} + 3\eta^{q^2+q}k + 2\eta^{2q^2}k  \big)y_1^3 + \big( \beta^{2}\eta^{q+1} + 6\beta^{q+1}\eta^{\frac{q^2+2q+1}{2}}k + 4\beta^{2q^2+1}\eta^{\frac{2q^2+q+1}{2}}k \\
& - 2\beta^{q^3+q^2+1}\eta^{\frac{q^2+q+2}{2}} - 8\beta\eta^{\frac{2q^2+q+1}{2}}k + \beta^{2q}\eta^{q^2+q}k^2 + 6\beta^{2q^2+q}\eta^{\frac{3q^2+q}{2}}k^2 - 4\beta^{q^3+q^2+q}\eta^{\frac{2q^2+q+1}{2}}k - 2\beta^{q}\eta^{\frac{q^2+q+2}{2}}\\
& - 6\beta^{q}\eta^{\frac{3q^2+q}{2}}k^2 - 2\beta^{2q^2}\eta^{q^2+q}k^2 - 6\beta^{2q^2}\eta^{2q^2}k^2 + 2\beta^{q^3+q^2}\eta^{\frac{q^2+2q+1}{2}}k + \beta^{2q^3}\eta^{q^2+1} + \eta^{q+1} + \eta^{q^2+1} + 3\eta^{q^2+q}k^2 \\
& + 6\eta^{2q^2}k^2 \big)y_1^2 + 2\big( \beta^{2q^2+1}\eta^{\frac{2q^2+q+1}{2}}k^2 - \beta^{q^3+q^2+1}\eta^{\frac{q^2+q+2}{2}}k - 2\beta\eta^{\frac{2q^2+q+1}{2}}k^2 + \beta^{2q^2+q}\eta^{\frac{3q^2+q}{2}}k^3 \\
& - \beta^{q^3+q^2+q}\eta^{\frac{2q^2+q+1}{2}}k^2 - \beta^{q}\eta^{\frac{q^2+q+2}{2}}k - \beta^{q}\eta^{\frac{3q^2+q}{2}}k^3 - 2\beta^{2q^2}\eta^{q^2}k^3 + \beta^{2q^3}\eta^{q^2+1}k + \eta^{q^2+1}k + 2\eta^{2q^2}k^3 \big)y_1 \\
& - \beta^{2q^2}\eta^{2q^2}k^4 + \beta^{2q^3}\eta^{q^2+1}k^2 + \eta^{q^2+1}k^2 + \eta^{2q^2}k^4.\\
\end{aligned}
\]
Since the fourth equation of \eqref{eq:sist} is linear in $y_3$, it is enough to prove that the first, second and third equation of \eqref{eq:sist} define an algebraic function field $\KK(y_1,w_0,w_1,w_2)/\KK$ of transcendence degree $1$ to have that $y_3\in\KK(y_1,w_0,w_1,w_2)$ and Lemma \ref{lemma:irreduciblecurve} is proved.

\subsubsection{The function field $\KK(y_1,w_0)$}\label{sec:primaest}
Define the polynomial
\[F_0(Y):=\eta^{q}Y^2 - 2\beta\eta^{\frac{q+1}{2}}Y+\eta=\eta^q(Y-\lambda_1)(Y-\lambda_2)\in\KK[Y],
\]
having distinct roots 
\[
\lambda_1=\eta^{\frac{1-q}{2}}(\beta-\sqrt{\beta^2-1}),\quad \lambda_2=\eta^{\frac{1-q}{2}}(\beta+\sqrt{\beta^2-1}).
\]
Then the first equation $w_0^2=F_0(y_1)$ of \eqref{eq:sist} defines a Kummer extension $\KK(y_1,w_0)/\KK(y_1)$ of degree $2$ with exactly two ramified places, namely the zeros of $y_1-\lambda_1$ and $y_1-\lambda_2$.

\subsubsection{The function field $\KK(y_1,w_0,w_1)$}\label{sec:secondaest}
Define the polynomial
\[
F_1(Y):= \left(\eta^{q}-2\beta\eta^{\frac{q+1}{2}}+\eta\right)^{q}Y^2 - 2k\left(\beta\eta^{\frac{q+1}{2}} - \eta^{q}\right)^{q}Y + k^2\eta^{q^2}\in\KK[Y].
\]
Up to this point, the only require on $\eta$ is that it is a non-square in $\mathbb{F}_{q^4}$, and hence we have $(q^4-1)/2$ choices for $\eta$.
The coefficient of $Y^2$ in $F_1(Y)$ vanishes for at most $q$ values of $\eta$.
Since $q<(q^4-1)/2$, we can choose $\eta$ as a non-square $\mathbb{F}_{q^4}$ such that the coefficient of $Y^2$ in $F_1(Y)$ does not vanish.
Now, choose $k\ne0$. Then $F_1(Y)$ has distinct roots
\[
\mu_1=k\cdot\left(\frac{-\eta^{q}+\eta^{\frac{q+1}{2}}(\beta-\sqrt{\beta^{2}-1})}{\eta^{q}-2\beta\eta^{\frac{q+1}{2}}+\eta}\right)^q,\qquad \mu_2=k\cdot\left(\frac{-\eta^{q}+\eta^{\frac{q+1}{2}}(\beta+\sqrt{\beta^{2}-1})}{\eta^{q}-2\beta\eta^{\frac{q+1}{2}}+\eta}\right)^q.
\]
Up to excluding four values of $k\in\KK$, we have $\{\lambda_1,\lambda_2\}\cap\{\mu_1,\mu_2\}=\emptyset$.
Thus, $F_1(y_1)$ has four simple zeros in $\KK(y_1,w_0)$, namely the zeros of $y_1-\mu_1$ and $y_1-\mu_2$.
Then the second equation $w_1^2=F_1(y_1)$ of \eqref{eq:sist} defines a Kummer extension $\KK(y_1,w_0,w_1)/\KK(y_1,w_0)$ of degree $2$, in which the ramified places are the zeros of $y_1-\mu_1$ and $y_1-\mu_2$.

Note that the ramified places in $\KK(y_1,w_0,w_1)/\KK(y_1)$ are the $8$ zeros of $y_1-\lambda_1$, $y_1-\lambda_2$, $y_1-\mu_1$ and $y_1-\mu_2$, all of them with ramification index $2$.

\subsubsection{The function field $\KK(y_1,w_0,w_1,w_2)$}
Define $u_2=w_2+\frac{b}{2a}$.
Then the third equation of \eqref{eq:sist} reads
\begin{eqnarray*}
u_2^2 &=& \frac{(b/2)^2-ac}{a^2}\\
&=& \left(\frac{ \beta^{q^3}\sqrt{\eta}  - \beta^{q^2}\sqrt{\eta}^{q^2}(y_1+k)}{a} \right)^2 \left(w_0w_1 + A -\eta^{\frac{q^2+1}{2}}\gamma (y_1+k) \right)\left(w_0w_1 + A +\eta^{\frac{q^2+1}{2}}\gamma (y_1+k)\right)
\end{eqnarray*}
where
\[
A= -\beta\eta^{\frac{q+1}{2}}y_1 + \beta^{q^3+q^2}\eta^{\frac{q^2+1}{2}}(y_1+k) + \eta^{q}y_1^2 - \beta^{q}\eta^{\frac{q^2+q}{2}}y_1(y_1+k), \qquad \gamma^2 = (\beta^{2q^3}-1)(\beta^{2q^2}-1).
\]
Write $u_2^2=\chi^2\cdot\xi\cdot\zeta$, where 
\[
\xi= w_0w_1 + A -\eta^{\frac{q^2+1}{2}}\gamma (y_1+k)=: w_0w_1+ \bar{a}y_1^2 + \bar{b}y_1 + \bar{c},
\]
\[\zeta=\xi+2\eta^{\frac{q^2+1}{2}}\gamma (y_1+k)= w_0w_1 + A +\eta^{\frac{q^2+1}{2}}\gamma (y_1+k)=w_0w_1+ \bar{a}y_1^2 + (\bar{b}+2\eta^{\frac{q^2+1}{2}}\gamma)y_1 + \bar{c}+2\eta^{\frac{q^2+1}{2}}\gamma k, 
\]
with
\[\bar{a}=\eta^{q}-\beta^{q}\eta^{\frac{q^2+q}{2}},\quad \bar{b}=-\beta^q\eta^{\frac{q^2+q}{2}}k - \beta\eta^{\frac{q+1}{2}} + \beta^{q^3+q^2}\eta^{\frac{q^2+1}{2}}-\eta^{\frac{q^2+1}{2}}\gamma,\quad \bar{c}=k(\beta^{q^3+q^2}\eta^{\frac{q^2+1}{2}} - \eta^{\frac{q^2+1}{2}}\gamma).
\]
In order to prove that the third equation of \eqref{eq:sist} defines a Kummer extension $\KK(y_1,w_0,w_1,w_2)/\KK(y_1,w_0,w_1)$ of degree $2$, we show that $\xi\cdot\zeta$ is not a square in $\KK(y_1,w_0,w_1)$.

Define $\theta:= w_0w_1-(\bar{a}y_1^2 + \bar{b}y_1 + \bar{c})$.
\begin{itemize}
    \item[(i)] We prove that the functions $\xi,\zeta,\theta$ have exactly four poles with valuation $-2$ at each of them, and therefore have exactly eight zeros, counted with multiplicity.
    In fact, the poles of $\xi,\zeta,\theta$ are among the four simple poles $P_{\infty}^i$ of $y_1$ in $\KK(y_1,w_0,w_1)$.
    Clearly $1/y_1$ is a local parameter at $P_{\infty}^i$. By Sections \ref{sec:primaest} and \ref{sec:secondaest}, the first term of the Laurent series expansion of $w_0w_1$ at $P_{\infty}^i$ with respect to $1/y_1$ is $\rho y_1^2$, where $\rho^2=\eta^q(\eta^{q^2}-2\beta^{q}\eta^{\frac{q^2+q}{2}}+\eta^q)$. If $\rho=\bar{a}$ or $\rho=-\bar{a}$, then $\rho^2=\bar{a}^2$ which implies $\beta^{2q}=1$, a contradiction.
    \item[(ii)] We prove that, up to excluding a finite number of $k\in\KK$, the zeros of $\xi\cdot\theta$ in $\KK(y_1,w_0,w_1)$  are unramified over $\KK(y_1)$.
    Suppose by contradiction that a zero $P$ of $\xi\cdot\theta$ is also a zero of $y_1-\lambda_1$ or $y_1-\lambda_2$. Then $w_0(P)=0$, and the resultant $R(k)$ between $F_0(Y)$ and $\bar{a}Y^2+\bar{b}Y+\bar{c}$ with respect to $Y$ vanishes. $R(k)$ is a polynomial in the variable $k$ whose coefficients must vanish, otherwise it is enough to exclude $\deg_k R(k)$ values of $k\in\KK$ to get a contradiction. This gives
        \[
    	2\gamma\beta^{q+1} - 2\gamma\beta^{q^3+q^2} + 2\beta^{q+2}\eta^{\frac{q-q^2}{2}} - \beta^2\eta^{q-q^2} - 2\beta^{q^3+q^2+q+1} + \beta^{q} - 2\beta^{q}\eta^{\frac{q-q^2}{2}} + 2\beta^{2q^3+2q^2} - \beta^{2q^2} - \beta^{2q^3} + \eta^{q-q^2} + 1=0,
\]
\[
        2\gamma\beta^{q+1} - 2\gamma\beta^{q^3+q^2} - 2\beta^{q^3+q^2+q+1} + \beta^{2q} + 2\beta^{2q^3+2q^2} - \beta^{2q^2} - \beta^{2q^3} + 1=0,
\]
\begin{eqnarray*}
        &&2\gamma \beta^{q+2}\eta^{\frac{q^2+1}{2}} - \gamma \beta^2 \eta^{\frac{q+1}{2}} - 2\gamma \beta^{q^3+q^2+1} \eta^{\frac{q^2+1}{2}} + \gamma \eta^{\frac{q+1}{2}} - 2\beta^{q^3+q^2+q+2}\eta^{\frac{q^2+1}{2}} + \beta^{q^3+q^2+2}\eta^{\frac{q+1}{2}} \\&&
            +\beta^{2q+1}\eta^{\frac{q^2+1}{2}} + 2\beta^{2q^3+2q^2+1}\eta^{\frac{q^2+1}{2}} - \beta^{2q^2+1}\eta^{\frac{q^2+1}{2}} - \beta^{2q^3+1}\eta^{\frac{q^2+1}{2}} + \beta \eta^{\frac{q^2+1}{2}} - \beta^{q^3+q^2}\eta^{\frac{q+1}{2}}=0.
\end{eqnarray*}
    After eliminating $\gamma$, these two equations yield
    \begin{equation}\label{eq:condbeta}
    \beta^{2q}-\beta^{2q^2}-\beta^{2q^3}+1=0.
    \end{equation}
    Equation \eqref{eq:condbeta}, together with its conjugates
    \[
    \beta^{2q^2}-\beta^{2q^3}-\beta^{2}+1=0,\qquad
    \beta^{2q^3}-\beta^{2}-\beta^{2q}+1=0,\qquad
    \beta^{2}-\beta^{2q}-\beta^{2q^2}+1=0
    \]
    under the $q$-Frobenius map, provides a contradiction to $\beta\ne\pm1$. 

    The case of a zero $P$ of $\xi\cdot\theta$ being also a zero of $y_1-\mu_1$ or $y_1-\mu_2$ is ruled out analogously through the coefficients of the resultant between $F_1(Y)$ and $\bar{a}Y^2+\bar{b}Y+\bar{c}$ seen as a polynomial in $k$, whose vanishing provides a contradiction to $\beta\ne\pm1$. 

    \item[(iii)] Let $F_2(Y)\in\KK[Y]$ be defined by
    \[
\xi\cdot\theta=w_0^2w_1^2 - (\bar{a}y_1^2 + \bar{b}y_1 + \bar{c})^2 = c_4 y_1^4 + c_3 y_1^3 + c_2 y_1^2 + c_1 y_1 + c_0 =: F_2(y_1).
\]
By direct computations we have 
\begin{eqnarray*}
c_4 &=& -\eta^{q^2+q}(\beta^{2q}-1),\\
 c_3 &=& - 2\eta^{q^2+q}[(\beta^{2q}-1)k+\eta^{\frac{1-q}{2}}(\beta^{q}\gamma+\beta-\beta^{q^3+q^2+q})+\eta^{\frac{1-q^2}{2}}(\beta^{q^3+q^2}-\beta^{q+1}-\gamma)],\\
c_2&=&- \eta^{q^2+q}(\beta^{2q}-1)k^2 - 2\eta^{q^2+q}[\eta^{\frac{1-q}{2}}(2\gamma\beta^{q}+2\beta-2\beta^{q^3+q^2+q})+\eta^{\frac{1-q^2}{2}}(\beta^{q^3+q^2}-\beta^{q+1}-\gamma)] k\\
 &&- \eta[2\gamma\beta\eta^{\frac{q^2+q}{2}}-2\gamma\beta^{q^3+q^2}\eta^{q^2}+\beta^2\eta^{q}-2\beta^{q^2+q^2+1}\eta^{\frac{q^2+q}{2}}+2\beta^{q}\eta^{\frac{q^2+q}{2}}+2\beta^{2q^3+2q^2}\eta^{q^2}-\beta^{2q^2}\eta^{q^2}-\beta^{2q^3}\eta^{q^2}-\eta^{q}],\\
c_1 &=& - 2\eta^{q^2+1}k [\eta^{\frac{q-1}{2}}(\gamma\beta^q+\beta-\beta^{q^3+q^2+q})k+\eta^{\frac{q-q^2}{2}}(\gamma\beta-\beta^{q^3+q^2+1}+\beta^q)-2\gamma\beta^{q^3+q^2}+2\beta^{2q^3+2q^2}-\beta^{2q^2}-\beta^{2q^3}],\\
c_0 &=& - \eta^{q^2+1}(\gamma-\beta^{q^3+q^2}-1)(\gamma-\beta^{q^3+q^2}+1)k^2.
\end{eqnarray*}
Note that $F_2(Y)$ has degree $4$, since $c_4=- \eta^{q^2+q}(\beta^{2q}-1)\neq 0$.

We prove that, if $F_2(Y)$ is not a square in $\KK[Y]$, then $\xi\cdot\zeta$ is not a square in $\KK(y_1,w_0,w_1)$.
In fact, if $F_2(Y)$ is not a square in $\KK[Y]$, then $F(Y)$ has two distinct roots $\rho_1,\rho_2\in\KK$ with odd multiplicity. Then, for any $i=1,2$, $\xi$ has two unramified zeros over $\rho_i$, with valuation equal to the odd multiplicity of $\rho_i$ as a zero of $F_2(Y)$. 
Suppose by contradiction that $\xi\cdot\zeta$ is a square, and hence has even valuation at every place. Then $\zeta$ has zeros over $\rho_1$ and $\rho_2$, and $\zeta-\xi=2\eta^{\frac{q^2+1}{2}}\gamma (y_1+k)$ has zeros over $\rho_1$ and $\rho_2$, a contradiction to $\rho_1\ne\rho_2$.

\item[(iv)] We prove that $F_2(Y)$ is not a square in $\KK[Y]$.
Suppose by contradiction that
\[
F_2(Y)=(d_2Y^2+d_1Y+d_0)^2
\]
for some $d_i\in\KK$, that is,
\[
d_0^2=c_0,\quad 2d_0d_1=c_1,\quad d_1^2+2d_0d_1=c_2,\quad 2d_1d_2=c_3,\quad d_2^2=c_4.
\]
This implies
\begin{equation}\label{eq:daje}
4c_0(c_2-c_1)=c_1^2, \quad 4c_4(c_2-c_1)=c_3^2.
\end{equation}

Equation \eqref{eq:daje} provides two polynomials in $k$ that must be identically zero, namely $p_1(k):=c_1^2-4c_0(c_2-c_1)$ and $p_2(k):=c_3^2-4c_4(c_2-c_1)$, otherwise a contradiction is obtained by excluding a finite number of $k\in\KK$. 
By direct computation, the coefficient of $k^2$ in $p_2(k)$ is $\beta^q\gamma + \beta - \beta^{q^3+q^2+q},$ and hence
\begin{equation}\label{eq:1}
\beta^q\gamma + \beta - \beta^{q^3+q^2+q}=0.
\end{equation}
On the other hand, the vanishing of the coefficient of $k^2$ in $p_1(k)$ yields
\begin{equation}\label{eq:2}
4\eta^{q^2+q+2}(\gamma^2 + 2\gamma\beta^{q+1} - 2\gamma\beta^{q^3+q^2} + \beta^2 - 2\beta^{q^3+q^2+q+1} + \beta^{2q} + \beta^{2q^3+2q^2} - 1)=0.
\end{equation}
The elimination of $\gamma$ through the resultant between the left-hand sides in Equations \eqref{eq:1} and \eqref{eq:2} yields $\beta^{2q}-\beta^2=0$. This is a contradiction to the hypothesis of Lemma \ref{lemma:irreduciblecurve}.

\end{itemize}

\subsection{Proof of Lemma \ref{lemma:reduciblecurve}}\label{sec:lemmaqodd2}

Note that $\beta^{q-1}\in\{1,-1\}$, and that $\mathcal{V}$ can be equivalently defined by the following linear combinations of the equations in \eqref{eq:V}: \begin{equation}\label{eq:V2}
\mathcal{V}\colon\begin{cases}
W_0^2=\eta Y_0^2 + \eta^{q} Y_1^2 - 2\beta\eta^{\frac{q+1}{2}}Y_0 Y_1,\\
(W_1-W_0)(W_1+W_0)=-\eta(Y_0 - \beta^{q-1}\eta^{\frac{q^2-1}{2}}Y_2)(Y_0 - 2\beta\eta^{\frac{q-1}{2}}Y_1 + \beta^{q-1}\eta^{\frac{q^2-1}{2}}Y_2),\\
(W_2-W_1)(W_2+W_1)=-\eta^q(Y_1-\beta^{q-1}\eta^{\frac{q^3-q}{2}}Y_3)(Y_1 - 2\beta\eta^{\frac{q^2-q}{2}}Y_2 + \beta^{q-1}\eta^{\frac{q^3-q}{2}}Y_3),\\
(W_0+W_1)(W_1+W_2)= \beta\eta^{\frac{q+1}{2}}(Y_1-\beta^{q-1}\eta^{\frac{q^3-q}{2}}Y_3)(Y_0-\beta^{q-1}\eta^{\frac{q^2-1}{2}}Y_2).\\
\end{cases}
\end{equation}

\begin{itemize}
    \item First we show that the equations
\begin{equation}\label{eq:lacomp}
\begin{cases}
w_0^2 = \eta  + \eta^{q} y_1^2 - 2\beta\eta^{\frac{q+1}{2}}y_1,\\
w_1^2 = \eta^q y_1^2 + \eta^{q^2} y_2^2 - 2\beta^q\eta^{\frac{q^2+q}{2}}y_1 y_2,\\
y_3 = - B'/A',\\
w_2 = \beta\eta^{\frac{q+1}{2}}\frac{(y_1-\beta^{q-1}\eta^{\frac{q^3-q}{2}}y_3)(1-\beta^{q-1}\eta^{\frac{q^2-1}{2}}y_2)}{w_0+w_1}-w_1,\\
\end{cases}
\end{equation}
where
\[
    \begin{split}
    A' = & - 2\beta^{q-1}\eta^{\frac{q^3-1}{2}} w_0 w_1 + \beta^{q-1}(\beta^2-1)\eta^{\frac{q^3+1}{2}} + 2\beta^{q}\eta^{\frac{q^3+q}{2}}y_1 - 2\beta^{2}\eta^{\frac{q^3+q^2}{2}}y_2 \\
    & - 2\beta^{q-1}\eta^{\frac{q^3+2q-1}{2}}y_1^2  + 2\beta\eta^{\frac{q^3+q^2+q-1}{2}}y_1 y_2 + \beta^{q-1}(\beta^2-1)\eta^{\frac{q^3+2q^2-1}{2}}y_2^2,\\
    B' = & \, 2(\beta - \eta^{\frac{q-1}{2}}y_1 + \beta^{q}\eta^{\frac{q^2-1}{2}}y_2)w_0w_1 - (\beta^2+1)\eta^{\frac{q+1}{2}}y_1(1  +\eta^{q^2-1} y_2^2 ) + 2\beta^{q}\eta^{\frac{q^2+1}{2}}y_2 \\
    &  + 4\beta\eta^{q}y_1^2 - 6\beta^{q+1}\eta^{\frac{q^2+q}{2}}y_1 y_2
         + 2\beta\eta^{q^2}y_2^2  - 2\eta^{\frac{3q-1}{2}}y_1^3 + 4\beta^{q}\eta^{\frac{q^2+2q-1}{2}}y_1^2 y_2 ,
        \end{split}
\]
define an algebraic function field $\fqq(y_1,w_0,y_2,w_1)/\mathbb{F}_{q^4}$ of transcendence degree $2$ over $\fqq$, with constant field $\fqq$; this will imply that $\fqq(y_1,w_0,y_2,w_1,x_3)$ is a function field of three variables over $\fqq$ with constant field $\fqq$, where $x_3$ is transcendent over $\fqq(y_1,w_0,y_2,w_1)$.

As shown in Section \ref{sec:primaest}, the first equation of \eqref{eq:lacomp} defines a Kummer extension $\fqq(y_1,w_0)/\fqq(y_1)$ of degree $2$ and transcendence degree $1$ over $\fqq$ with constant field $\fqq$.
Let $y_2$ be transcendent over $\fqq(y_1,w_0)$.
Then, for some non-square $\eta$ of $\fqq$, the function $\eta^q y_1^2 + \eta^{q^2} y_2^2 - 2\beta^q\eta^{\frac{q^2+q}{2}}y_1 y_2$ is not a square in $\KK(y_1,w_0,y_2)$, because for some $k\in\KK$ its specialization with $y_2=y_1+k$ is not a square in $\KK(y_1,w_0)$ as shown in Section \ref{sec:secondaest}.
Thus, the second equation of \eqref{eq:lacomp} defines a Kummer extension $\fqq(y_1,w_0,y_2,w_1)/\fqq(y_1,w_0,y_2)$ of degree $2$ and transcendence degree $2$ over $\fqq$ with constant field $\fqq$.
The element $A'$ is non-zero (since it has degree $1$ as a polynomial in $w_1$ over $\KK(y_1,w_0,y_2)$), and hence the element $y_3\in\fqq(y_1,w_0,y_2,w_1)$ is well-defined by the third equation of \eqref{eq:lacomp}.
From $w_0+w_1\ne0$ follows that the element $w_2\in\fqq(y_1,w_0,y_2,w_1)$ is well-defined by the fourth equation of $\eqref{eq:lacomp}$.
The claim on the extension $\fqq(y_1,w_0)/\fqq(y_1)$ is proved.

\item We show that the intersection between $\mathcal{V}$ and the hyperplane $W_0+W_1=0$ contains exactly one absolutely irreducible component $\tilde{\mathcal{V}}$ of dimension $3$, which is also the unique absolutely irreducible component of dimension $3$ in the intersection between $\mathcal{V}$ and the hyperplane $Y_0-\beta^{q-1}\eta^{\frac{q^2-1}{2}}Y_2=0$. Moreover, $\tilde{\mathcal{V}}$ has degree $4$.

Indeed, suppose $W_0+W_1=0$. Then either $Y_1-\beta^{q-1}\eta^{\frac{q^3-q}{2}}Y_3=0$ or $Y_0-\beta^{q-1}\eta^{\frac{q^2-1}{2}}Y_2=0$.
    \begin{itemize}
        \item Suppose $Y_1-\beta^{q-1}\eta^{\frac{q^3-q}{2}}Y_3=0$. Then the solutions of \eqref{eq:V2} are given by the following systems:
\begin{small}
        \[
        \begin{cases}
        W_0^2=\eta Y_0^2 + \eta^{q} Y_1^2 - 2\beta\eta^{\frac{q+1}{2}}Y_0 Y_1,\\
        W_1=-W_0,\\
        Y_3=\beta^{q-1}\eta^{\frac{q-q^3}{2}}Y_1,\\
        Y_2=\beta^{q-1}\eta^{\frac{1-q^2}{2}}Y_0,\\
        W_2=W_0;\\
        \end{cases}
        \begin{cases}
        W_0^2=\eta Y_0^2 + \eta^{q} Y_1^2 - 2\beta\eta^{\frac{q+1}{2}}Y_0 Y_1,\\
        W_1=-W_0,\\
        Y_3=\beta^{q-1}\eta^{\frac{q-q^3}{2}}Y_1,\\
        Y_2=\beta^{q-1}\eta^{\frac{1-q^2}{2}}Y_0,\\
        W_2=-W_0;\\
        \end{cases}
        \]
                \[
        \begin{cases}
        W_0^2=\eta Y_0^2 + \eta^{q} Y_1^2 - 2\beta\eta^{\frac{q+1}{2}}Y_0 Y_1,\\
        W_1=-W_0,\\
        Y_3=\beta^{q-1}\eta^{\frac{q-q^3}{2}}Y_1,\\
        Y_2=-\beta^{q-1}\eta^{\frac{1-q^2}{2}}Y_0+2\beta^q\eta^{\frac{q-q^2}{2}}Y_1,\\
        W_2=W_0;\\
        \end{cases}
        \begin{cases}
        W_0^2=\eta Y_0^2 + \eta^{q} Y_1^2 - 2\beta\eta^{\frac{q+1}{2}}Y_0 Y_1,\\
        W_1=-W_0,\\
        Y_3=\beta^{q-1}\eta^{\frac{q-q^3}{2}}Y_1,\\
        Y_2=-\beta^{q-1}\eta^{\frac{1-q^2}{2}}Y_0+2\beta^q\eta^{\frac{q-q^2}{2}}Y_1,\\
        W_2=-W_0.\\
        \end{cases}
        \]
\end{small}       
For each of these four systems, the following holds. The first equation defines an absolutely irreducible curve in $W_0,Y_0,Y_1$, as shown in Section \ref{sec:primaest}. Each of the remaining four equations is linear and defines an absolutely irreducible curve with an additional indeterminate; in terms of function fields, $\KK(w_0,y_1)$ remains the same. Since the indeterminate $X_3$ does not appear, the system defines an absolutely irreducible surface contained in $\mathcal{V}\subset\mathbb{P}^7$, namely a cone with vertex $(0:0:0:1:0:0:0:0)$; its function field is $\KK(w_0,y_1,x_3)$, where $x_3$ is transcendental over $\KK(w_0,y_1)$.

        \item Suppose $Y_0-\beta^{q-1}\eta^{\frac{q^2-1}{2}}Y_2=0$.
        By direct computation, \eqref{eq:V} reads
\[        \tilde{\mathcal{V}}\colon\begin{cases}
W_0^2=\eta Y_0^2 + \eta^{q} Y_1^2 - 2\beta\eta^{\frac{q+1}{2}}Y_0 Y_1,\\
W_2^2=\eta Y_0^2 + \eta^{q^3}Y_3^2 - 2\beta^{q}\eta^{\frac{q^3+1}{2}}Y_0Y_3,\\
Y_2=\beta^{q-1}\eta^{\frac{1-q^2}{2}}Y_0,\\
W_1=-W_0.\\
\end{cases}
\]
An argument analogous to the one used in Section \ref{sec:lemmairred} shows that $\tilde{\mathcal{V}}$ is absolutely irreducible. In fact, we have
\[\eta Y_0^2 + \eta^{q} Y_1^2 - 2\beta\eta^{\frac{q+1}{2}}Y_0 Y_1=\eta^q(Y_1-\lambda_1 Y_0)(Y_1-\lambda_2 Y_0)
\]
with $\lambda_1\ne\lambda_2$. After the specialization $Y_3=hY_1$ with $h\in\KK$, write
\[
(\eta Y_0^2 + \eta^{q^3}Y_3^2 - 2\beta^{q}\eta^{\frac{q^3+1}{2}}Y_0Y_3)\mid_{Y_3=hY_1}=\eta^{q^3}(Y_3-h\mu_1Y_0)(Y_3-h\mu_3Y_0).
\]
For a suitable $h\in\KK$, the values $h\mu_1$ and $h\mu_2$ are different from each other and different from $\lambda_1$ and $\lambda_2$.
Therefore, recalling that the indeterminate $X_3$ does not appear explicitly, the intersection between $\tilde{\mathcal{V}}$ and the hyperplane $Y_3=hY_1$ is an absolutely irreducible surface, namely a cone with vertex $(0:0:0:1:0:0:0:0)$, and hence $\tilde{\mathcal{V}}$ is an absolutely irreducible variety with $\dim(\tilde{\mathcal{V}})=3=\dim(\mathcal{V})$.
Hence, $\tilde{\mathcal{V}}$ is an absolutely irreducible component of $\mathcal{V}$, and it is readily seen that $\tilde{\mathcal{V}}$ is the unique $3$-dimensional component of $\mathcal{V}$ contained in the hyperplane $Y_0-\beta^{q-1}\eta^{\frac{q^2-1}{2}}Y_2=0$.
Clearly $\tilde{\mathcal{V}}$ is a complete intersection, whence $\deg(\tilde{\mathcal{V}})=4$.
\end{itemize}

\item With the same arguments as in the previous point, it is easily seen that the intersection between $\mathcal{V}$ and the hyperplane $W_1+W_2=0$ contains exactly one absolutely irreducible component $\hat{\mathcal{V}}$ of dimension $3$, given by
\[        \hat{\mathcal{V}}\colon\begin{cases}
W_0^2=\eta Y_0^2 + \eta^{q} Y_1^2 - 2\beta\eta^{\frac{q+1}{2}}Y_0 Y_1,\\
W_1^2=\eta^q Y_1^2 + \eta^{q^2} Y_2^2 - 2\beta^q\eta^{\frac{q^2+q}{2}}Y_1 Y_2,\\
Y_3=\beta^{q-1}\eta^{\frac{q-q^3}{2}}Y_1,\\
W_2=-W_1,\\
\end{cases}
\]
which is also the unique absolutely irreducible component of dimension $3$ in the intersection between $\mathcal{V}$ and the hyperplane $Y_1-\beta^{q-1}\eta^{\frac{q^3-q}{2}}Y_3=0$.
Moreover, $\hat{\mathcal{V}}$ has degree $4$.

\item We show that $\mathcal{V}$ has exactly another $3$-dimensional absolutely irreducible component $\mathcal{V}^{\prime}$ other than $\tilde{\mathcal{V}}$ and $\hat{\mathcal{V}}$, having degree $8$.

Indeed, by what has been shown for $\tilde{\mathcal{V}}$, we can assume that $W_0+W_1$ does not vanish identically, so that the fourth equation of \eqref{eq:V2} yields
\[
w_2=\beta\eta^{\frac{q+1}{2}}\frac{(y_1-\beta^{q-1}\eta^{\frac{q^3-q}{2}}y_3)(y_0-\beta^{q-1}\eta^{\frac{q^2-1}{2}}y_2)}{w_0+w_1}-w_1.
\]
After replacing $w_2$ and deleting the denominator $w_0+w_1$, the third equation of \eqref{eq:V2} gives
\[
    (y_1 - \beta^{q-1}\eta^{\frac{q^3-q}{2}}y_3)\cdot (A\cdot y_3+B) =0,
    \]
    where
    \[
    \begin{split}
    A = & - 2\beta^{q-1}\eta^{\frac{q^3-1}{2}} w_0 w_1 + \beta^{q-1}(\beta^2-1)\eta^{\frac{q^3+1}{2}}y_0^2 + 2\beta^{q}\eta^{\frac{q^3+q}{2}}y_0 y_1 - 2\beta^{2}\eta^{\frac{q^3+q^2}{2}}y_0 y_2 \\
    & - 2\beta^{q-1}\eta^{\frac{q^3+2q-1}{2}}y_1^2  + 2\beta\eta^{\frac{q^3+q^2+q-1}{2}}y_1 y_2 + \beta^{q-1}(\beta^2-1)\eta^{\frac{q^3+2q^2-1}{2}}y_2^2,\\
    B = & \, 2(\beta y_0 - \eta^{\frac{q-1}{2}}y_1 + \beta^{q}\eta^{\frac{q^2-1}{2}}y_2)w_0w_1 - (\beta^2+1)\eta^{\frac{q+1}{2}}(y_0^2 y_1  +\eta^{q^2-1}y_1 y_2^2 ) + 2\beta^{q}\eta^{\frac{q^2+1}{2}}y_0^2 y_2 \\
    &  + 4\beta\eta^{q}y_0 y_1^2 - 6\beta^{q+1}\eta^{\frac{q^2+q}{2}}y_0 y_1 y_2
         + 2\beta\eta^{q^2}y_0 y_2^2  - 2\eta^{\frac{3q-1}{2}}y_1^3 + 4\beta^{q}\eta^{\frac{q^2+2q-1}{2}}y_1^2 y_2 .
        \end{split}
    \]
    By what has been shown for $\hat{\mathcal{V}}$, we can assume that $y_1 - \beta^{q-1}\eta^{\frac{q^3-q}{2}}y_3$ does not vanish identically, so that the third equation of \eqref{eq:V2} can be replaced by the equation $A\cdot y_3+B=0$.
    Therefore, every $3$-dimensional absolutely irreducible component of $\mathcal{V}$ other than $\tilde{\mathcal{V}}$ and $\hat{\mathcal{V}}$ is contained in the variety $\mathcal{V}^{\prime}$ whose function field is defined by
    \[
              \begin{cases}
        w_0^2=\eta y_0^2 + \eta^{q} y_1^2 - 2\beta\eta^{\frac{q+1}{2}}y_0 y_1,\\
        w_1^2 = \eta^q y_1^2 + \eta^{q^2} y_2^2 - 2\beta^q\eta^{\frac{q^2+q}{2}}y_1 y_2,\\
        y_3=-B/A,\\
        w_2= \beta\eta^{\frac{q+1}{2}}\frac{(y_1-\beta^{q-1}\eta^{\frac{q^3-q}{2}}y_3)(y_0-\beta^{q-1}\eta^{\frac{q^2-1}{2}}y_2)}{w_0+w_1}-w_1.\\
        \end{cases}
    \]
    Note that the system \eqref{eq:lacomp} defines exactly the coordinate functions of $\mathcal{V}^{\prime}$ after dehomogenizing the equations with respect to $Y_0$. Therefore, by what has been shown above for $\fqq(y_1,w_0,y_2,w_1,x_3)$, $\mathcal{V}^{\prime}$ is absolutely irreducible and has dimension $3$.
    Finally, we use the fact that the degree of $\mathcal{V}$ is the sum of the degrees if its absolutely irreducible components of maximal dimension (see \cite[Proposition 7.6 (b)]{Hart}) to conclude that \[\deg(\mathcal{V}^{\prime})=\deg(\mathcal{V})-\deg(\tilde{\mathcal{V}})-\deg(\hat{\mathcal{V}})=8.\]
\end{itemize}

\section{Conclusions and open problems}\label{sec:open}

In this paper we have considered for $n=8$ the rank metric codes $\mathcal{C}_{\delta,s}\subseteq\mathcal{L}_{n,q}$, where $\delta\in\mathbb{F}_{q^n}^*$ and $s$ is coprime with $n/2$.
We have given a partial answer to Conjecture 4.6 in \cite{PZZ}, proving its validity under the assumption that $q$ is odd and $q\geq1039891$.

The following questions naturally arise, and we list them as open problems. A solution to them would complete the classification of MRD codes $\mathcal{C}_{\delta,s}$ in $\mathcal{L}_{n,q}$ for any even positive integer $n$ and prime power $q$.

\begin{enumerate}
    \item Let $n\geq10$, and suppose that $n<4s+1$, so that the assumptions of \cite[Theorem 4.5]{PZZ} do not hold. Is it still possible to classify the codes $\mathcal{C}_{\delta,s}\subseteq\mathcal{L}_{n,q}$ which are MRD?
    \item Let $n=8$ and $q$ be an odd prime power with $q<1039891$. Classify the codes $\mathcal{C}_{\delta,s}\subseteq\mathcal{L}_{8,q}$ which are MRD: is it true that $\mathcal{C}_{\delta,s}$ is MRD if and only if $\mathrm{N}_{q^8/q^4}(\delta)=-1$? For $q\leq11$ this is true; see \cite[Remark 7.4]{CMPZ}.
    \item Let $n=8$ and $q$ be an even prime power. Classify the codes $\mathcal{C}_{\delta,s}\subseteq\mathcal{L}_{8,q}$ which are MRD: is it true that $\mathcal{C}_{\delta,s}$ is never MRD? For $q\leq8$ this is true; see \cite[Remark 7.4]{CMPZ}.
\end{enumerate}

In order to deal with the point (3), we have tried to apply algebraic geometric techniques, similar to the ones used in this paper. More precisely, a result analogous to Proposition \ref{prop:condX} holds for $q$ even after replacing the curve $\mathcal{X}_{\delta,s}$ in \eqref{eq:curveqodd_withs} with a suitable plane curve $\mathcal{X}_{\delta,s}^{even}$, as shown in \cite[Section 3.2]{PZZ}.
In analogy with Equation \eqref{eq:espressioni} and the variety $\mathcal{V}$ in this paper, we have then considered the $q$-powers of the indeterminates of $\mathcal{X}_{\delta,s}^{even}$ as the indeterminates of a higher-dimensional variety $\mathcal{V}_{even}$.
However, we have not been able to decide whether $\mathcal{V}_{even}$ contains an absolutely irreducible rational component, although computational experiments for small values of $q$ suggest that $\mathcal{V}_{even}$ may be absolutely irreducible.

\section{Acknoledgements}

The authors thank the anonymous referees for the valuable comments that helped to improve the paper.  The first author was funded by the project ``Metodi matematici per la firma digitale ed il cloud computing" (Programma Operativo Nazionale (PON) “Ricerca e Innovazione” 2014-2020, University of Perugia).
The second author was funded by the project ``Attrazione e Mobilità dei
Ricercatori'' Italian PON Programme (PON-AIM 2018 num. AIM1878214-2).
This research was supported by the project ``VALERE: VAnviteLli pEr la RicErca" of the University of Campania ``Luigi Vanvitelli'', and by the Italian National Group for Algebraic and Geometric Structures and their Applications (GNSAGA - INdAM).
Data sharing is not applicable to this article as no datasets were generated or analysed during the current study.

\end{document}